\documentclass[reqno,a4paper, 11pt]{amsart}

\usepackage[ansinew]{inputenc}
\usepackage{amsfonts}
\usepackage{amsmath}
\usepackage{amssymb}
\usepackage{url}
\usepackage[pdftex,hyperindex]{hyperref}

\newtheorem{theorem}{Theorem}
\newtheorem{lemma}[theorem]{Lemma}
\newtheorem{corollary}[theorem]{Corollary}

\newtheorem{oldtheorem}{Theorem}

\theoremstyle{definition}

\newtheorem{question}{Question}

\theoremstyle{remark}
\newtheorem{remark}{Remark}
\newtheorem{example}{Example}

\numberwithin{equation}{section}

\newcommand{\set}[1]{\left\{#1\right\}}
\newcommand{\abs}[1]{\left|#1\right|}
\newcommand{\brac}[1]{\left(#1\right)}

\newcommand{\nm}[1]{\lVert#1\rVert}
\newcommand{\bnm}[1]{\big\Vert#1\big\Vert}

\renewcommand{\L}[1]{\mathcal{L}^{#1}}
\newcommand{\D}{\mathbb{D}}

\newcommand{\RR}{\mathcal{R}}
\newcommand{\B}{\mathcal{B}}
\newcommand{\N}{\mathbb{N}}

\newcommand{\HH}{\mathcal{H}}
\newcommand{\C}{\mathbb{C}}

\renewcommand{\phi}{\varphi}

\newcommand{\T}{\mathbb{T}}
\newcommand{\BMOA}{\rm BMOA}
\newcommand{\LMOA}{\rm LMOA}
\newcommand{\VMOA}{\rm VMOA}

       \def\b{\beta}        
           
     \def\om{\omega}      
       \def\t{\theta}       
                  \def\z{\zeta}
                  \def\vp{\varphi}

\begin{document}

\title[Linear differential equations with slowly growing solutions]
{Linear differential equations with\\ slowly growing solutions}

\author{Janne Gr\"ohn, Juha-Matti Huusko and Jouni R\"atty\"a}
\date{\today}
\address{Department of Physics and Mathematics, University of Eastern Finland,\newline
\indent P.O.
Box 111, FI-80101 Joensuu, Finland}
\email{janne.grohn@uef.fi}

\address{Department of Physics and Mathematics, University of Eastern Finland,\newline
\indent P.O.
Box 111, FI-80101 Joensuu, Finland}
\email{juha-matti.huusko@uef.fi}

\address{Department of Physics and Mathematics, University of Eastern Finland,\newline
\indent P.O.
Box 111, FI-80101 Joensuu, Finland}
\email{jouni.rattya@uef.fi}

\thanks{The first author is supported in part by the Academy of Finland \#286877; the second author is supported in part by the Academy of Finland \#268009, and the Faculty of Science and Forestry of the University of Eastern Finland \#930349; and the third
author is supported in part by the Academy of Finland \#268009, the
Faculty of Science and Forestry of University of Eastern Finland \#930349,
La Junta de Andaluc\'ia (FQM210) and (P09-FQM-4468),
and the grants MTM2011-25502, MTM2011-26538 and MTM2014-52865-P}

\keywords{Growth of solution, Hardy space, linear differential equation}
\subjclass[2010]{Primary 30H10, 34M10}


\begin{abstract}
This research concerns coefficient conditions for linear differential equations in the unit disc
of the complex plane.
In the higher order case the separation of zeros (of maximal multiplicity) of solutions
is considered, while in the second order case  slowly growing solutions in $H^\infty$, 
$\rm{BMOA}$ and the Bloch space are discussed. A counterpart of the 
Hardy-Stein-Spencer formula for higher derivatives
is proved, and then applied to study solutions in the Hardy spaces.
\end{abstract}

\maketitle


\section{Introduction}

A fundamental question in the study of complex linear differential equations with analytic
coefficients in a complex domain is to relate the growth of coefficients to the growth
of solutions  and to the distribution of their zeros. In the case of fast growing solutions, Nevanlinna and
Wiman-Valiron theories have turned out to be very useful both in
the unit disc \cite{FR:2010, L:2008} and in the
complex plane \cite{L:1993, L:2008}.

We restrict ourselves to the case of the unit disc $\D = \{z\in\C : |z|<1 \}$.
In addition to methods above, theory of conformal maps has been
used to establish interrelationships between the growth of coefficients and the geometric
distribution (and separation) of zeros of solutions. This connection was one of the highlights
in Nehari's seminal paper~\cite{N:1949},
according to which a sufficient condition for the injectivity of a locally univalent meromorphic
function can be given in terms of its Schwarzian derivative. In the setting of differential
equations, Nehari's theorem \cite[Theorem~I]{N:1949} admits the following (equivalent) formulation:
if $A$ is analytic in $\D$ and
\begin{equation} \label{eq:neh}
  \sup_{z\in\D}\, |A(z)| (1-|z|^2)^2
\end{equation}
is at most one, then each non-trivial solution of
\begin{equation}\label{eq:de1}
  f''+Af=0
\end{equation}
has at most one zero in $\D$. Few years later, Schwarz showed \cite[Theorems~3--4]{S:1955}
that if $A$ is analytic in $\D$ then
zero-sequences of all non-trivial solutions of \eqref{eq:de1} are
separated in the hyperbolic metric if and only if \eqref{eq:neh} is finite.
The necessary condition, corresponding to Nehari's theorem, was given by Kraus~\cite{K:1932}.
For more recent developments based on localization of the classical results, see~\cite{CGHR:2013}.
In the case of higher order linear differential equations
\begin{equation} \label{eq:dek}
  f^{(k)}+A_{k-1}f^{(k-1)}+\cdots+A_1f'+A_0f=0, \quad k\in\N,
\end{equation}
with analytic coefficients $A_0,\dotsc, A_{k-1}$, 
this line of reasoning has not given complete results.
Some progress on the subject was obtained
by Kim and Lavie in seventies and eighties, among many other authors.

Nevanlinna and Wiman-Valiron theories, in the form they are known today,
are not sufficiently delicate tools to study slowly growing solutions of \eqref{eq:de1},
and hence different approach must be employed.
An~important breakthrough in this regard was \cite{P:1982},
where Pommerenke obtained a sharp sufficient condition for the analytic coefficient $A$
which places all solutions $f$ of \eqref{eq:de1} to the classical Hardy space $H^2$.
Pommerenke's idea was to use Green's formula twice to write the $H^2$-norm of $f$ in terms of
$f''$, employ the differential equation \eqref{eq:de1}, and then apply Carleson's theorem for
the Hardy spaces \cite[Theorem~9.3]{D:1970}. Consequently, the coefficient condition was given
in terms of Carleson measures. The leading idea of this (operator theoretic) approach
has been extended to study, for example,
solutions in the Hardy spaces \cite{R:2007}, Dirichlet type spaces \cite{HKR:2007} and
growth spaces \cite{GR:2016, HKR:2016}, to name a few instances.

Our intention is to establish sufficient conditions for the coefficient of \eqref{eq:de1}
which place all solutions to $H^\infty$, $\BMOA$ or to the Bloch space.
In principle, Pommerenke's original idea could be modified to cover these cases, but in practice,
this approach falls short
since either it is difficult to find a useful expression for the norm
in terms of the second derivative (in the case of $H^\infty$)
or the characterization of Carleson measures is not known
(in the cases of $\rm{BMOA}$ and Bloch). Concerning Carleson measures for the 
Bloch space, see \cite{GPP-GR:2008}. Curiously enough, the best known coefficient condition placing all solutions of \eqref{eq:de1} in the
Bloch space is obtained by  straightforward integration \cite{HKR:2016}.
Our approach takes advantage of the reproducing formulae,
and is different to ones in the literature.


\section{Main results}

Let $\mathcal{H}(\D)$ denote the collection
of functions analytic in $\D$, and let $m$
be the Lebesgue area measure, normalized so that $m(\D)=1$.
By postponing the rigorous definitions to the forthcoming sections,
we proceed to outline our results. We begin with the zero distribution of 
non-trivial solutions of the linear differential equation
\begin{equation} \label{eq:de3}
  f'''+A_{2}f''+A_1f'+A_0f=0
\end{equation}
with analytic coefficients. Note that zeros of non-trivial solutions of \eqref{eq:de3} are 
at most two-fold. Let $\vp_a(z) = (a-z)/(1-\overline{a} z)$, for $a,z\in\D$,
denote an~automorphism of $\D$ which coincides with its own inverse.


\begin{theorem}\label{Theorem:SchwarzForHigherOrder1}
Let $f$ be a non-trivial solution of \eqref{eq:de3}
where $A_0,A_1,A_2\in\mathcal{H}(\D)$.
\begin{itemize}
\item[\rm (i)]
  If
  \begin{equation} \label{eq:assumpp0}
    \sup_{z\in\D} \,  |A_j(z)| (1-|z|^2)^{3-j}< \infty, \quad j=0,1,2,
  \end{equation}
  then the sequence of two-fold zeros of $f$
  is a finite union of separated sequences.
  
\item[\rm (ii)]
  If
  \begin{equation} \label{eq:assumpp}
    \sup_{a\in\D} \, \int_{\D} |A_j(z)| (1-|z|^2)^{1-j} \big( 1- |\varphi_a(z)|^2 \big) \, dm(z)< \infty, \quad j=0,1,2,
  \end{equation}
  then the sequence of two-fold zeros of $f$
  is a finite union of uniformly separated sequences.
\end{itemize}
\end{theorem}

Theorem~\ref{Theorem:SchwarzForHigherOrder1}(i) should be compared to
the second order case \cite[Theorem~3]{S:1955}, which was already mentioned in the introduction.
For the counterpart of
Theorem~\ref{Theorem:SchwarzForHigherOrder1}(ii), see \cite[Theorem~1]{GNR:preprint}.
The proof of Theorem~\ref{Theorem:SchwarzForHigherOrder1} is presented in Section~\ref{sec:sproof},
and it is based on a~conformal transformation of \eqref{eq:de3}, Jensen's formula,
and on a sharp growth estimate for solutions of \eqref{eq:de3}.
Theorem~\ref{Theorem:SchwarzForHigherOrder1} extends to
the case of higher order differential equations~\eqref{eq:dek},
but we leave details for the interested reader.

The following results concern
slowly growing solutions of the second order differential equation \eqref{eq:de1}.
A sufficient condition for the analytic coefficient $A$, which forces all solutions 
of \eqref{eq:de1} to be bounded, is given in terms of Cauchy transforms.
The space $\mathcal{K}$ of Cauchy transforms consists of functions in $\mathcal{H}(\D)$ 
that take the form $\int_{\partial\D} (1-\overline{\z}z)^{-1} \, d\mu(\zeta)$,
where $\mu$ is a finite, complex, Borel measure on the unit circle $\partial\D$. For more details we refer to
Section~\ref{sc:bounded}, where the following theorem is proved.


\begin{theorem} \label{th:bounded}
Let $A\in\mathcal{H}(\D)$. If
$\limsup\limits_{r\to 1^-}\,\sup\limits_{z\in\D}\, \nm{A_{r,z}}_{\mathcal{K}}<1$ for
\begin{equation*}
  A_{r,z}(u)=\overline{\int_0^z\int_0^\zeta\frac{A(rw)}{1-\overline{u}w}\,dw\,d\zeta},\quad u\in\D,
\end{equation*}
then all solutions $f$ of~\eqref{eq:de1} are bounded.
\end{theorem}

The question converse to Theorem~\ref{th:bounded} is open and
appears to be difficult.
The boundedness of one non-trivial solution of \eqref{eq:de1} is not enough to guarantee that \eqref{eq:neh}
is finite, which can be easily seen by considering the solution $f(z)=\exp(-(1+z)/(1-z))$ of \eqref{eq:de1}
for $A(z)=-4z/(1-z)^4$, $z\in\D$. However, if \eqref{eq:de1} admits linearly independent
solutions $f_1,f_2\in H^\infty$ such that $\inf_{z\in\D} \big( |f_1(z)|+|f_2(z)| \big) >0$,
then \eqref{eq:neh} is finite. This is a consequence of the Corona theorem~\cite[Theorem~12.1]{D:1970}, 
according to which there exist $g_1,g_2\in H^\infty$ such that $f_1g_1+f_2g_2\equiv 1$,
and consequently $A=A+(f_1g_1+f_2g_2)''=2(f_1'g_1'+f_2'g_2')+f_1g_1''+f_2g_2''$.

We proceed to consider $\BMOA$, which contains
those functions in the Hardy space $H^2$ whose boundary values are of bounded mean oscillation.
The following result should be compared to \cite[Theorem~2]{P:1982}
as $\BMOA$ is a~conformally invariant subspace of $H^2$.


\begin{theorem}
\label{th:bmoa2}
Let $A\in\mathcal{H}(\D)$. If
\begin{equation}
  \label{eq:bmoacondnew}
  \sup_{a\in\D}\,
  \left( \log \frac{e}{1-|a|}  \right)^2
  \int_{\D} |A(z)|^2(1-|z|^2)^2 (1-|\varphi_a(z)|^2)  \, dm(z)
\end{equation}
is sufficiently small, then all solutions $f$ of~\eqref{eq:de1} satisfy $f\in\BMOA$.
\end{theorem}

To the best of our knowledge $\BMOA$ solutions of \eqref{eq:de1} have not been discussed in the literature before.
By \cite[Lemma~5.3]{PR:2014} or \cite[Theorem~1]{Z:2003}, 
\eqref{eq:bmoacondnew} is comparable to 
\begin{equation}
  \label{eq:bmoacond2}
  \sup_{a\in\D}\,
  \frac{\big( \log \frac{e}{1-|a|}\big)^2}{1-|a|}
  \int_{S_a} |A(z)|^2(1-|z|^2)^3 \, dm(z),
\end{equation}
where $S_a=\{re^{i\theta} : |a|<r<1,\,|\theta-\arg(a)|\leq (1-|a|)/2\}$ denotes the Carleson square 
with respect to $a\in\D\setminus\set{0}$ and $S_0=\D$. See also \cite[Lemma~3.4]{SZ:1999}.
Solutions in $\VMOA$, the closure of polynomials in $\BMOA$, are discussed in Section~\ref{Section:BMOA}
in which Theorem~\ref{th:bmoa2} is proved.

The case of the Bloch space $\mathcal{B}$ is especially interesting.
For $0< \alpha < \infty$, let $\L{\alpha}$ be the collection of those $A\in\mathcal{H}(\D)$ for which
\begin{equation*} \label{eq:hkrviiteold}
        \nm{A}_{\L{\alpha}} = \sup_{z\in\D}\, |A(z)| (1-|z|^2)^2 \left( \log\frac{e}{1-|z|} \right)^\alpha < \infty.
\end{equation*}
The comparison between $H^\infty_2$, $\L{\alpha}$ and the functions for which \eqref{eq:bmoacondnew}
is finite is presented in Section~\ref{sec:comparison}. It is known that, if $A\in \L{1}$
with sufficiently small norm, then all solutions of~\eqref{eq:de1} satisfy $f\in\mathcal{B}$.
This result was recently discovered with the best possible upper bound for $\nm{A}_{\L{1}}$
in \cite[Corollary~4(b) and Example~5(b)]{HKR:2016}. Actually, if $\nm{A}_{\L{1}}$ is sufficiently
small, then all solutions of \eqref{eq:de1} satisfy $f\in\mathcal{B} \cap H^2$ by \cite[Corollary~1]{P:1982}.
We point out that, if $A\in\L{\alpha}$ for any $1<\alpha<\infty$,
then all solutions of~\eqref{eq:de1} are bounded by~\cite[Theorem~G(a)]{HKR:2004}.
Solutions in the little Bloch space $\mathcal{B}_0$, the closure of polynomials in $\mathcal{B}$, 
are discussed in Section~\ref{Section:Bloch}, among other Bloch results.

The proof of Theorem~\ref{th:bounded}
is based on an application of the reproducing formula for $H^1$ functions, and it is natural to ask
whether this method extends to the cases of $\mathcal{B}$ and $\BMOA$.
In the case of $\mathcal{B}$, by using the reproducing formula for weighted Bergman spaces,
we prove a result (namely, Theorem~\ref{th:bloch-general}) offering a family of coefficient conditions,
which are given in terms of Bergman spaces with regular weights.
The case of $\BMOA$, by using the reproducing formula for $H^1$, is 
further considered in Section~\ref{sec:bmoa_alt}.

A careful reader observes that the results above are closely related to operator theory. 
Actually, if $f$ is a solution of \eqref{eq:de1}, then
\begin{equation} \label{eq:diffrep}
  f(z)=-\int_0^z\left(\int_0^\z f(w)A(w)\,dw\right) d\z+f'(0)z+f(0),\quad z\in\D.
\end{equation}
If we denote
\begin{equation*}
  S_A(f)(z)=\int_0^z\left(\int_0^\z f(w)A(w)\,dw\right) d\z, \quad z\in\D,
\end{equation*}
we obtain an integral operator, induced by the symbol $A\in\mathcal{H}(\D)$, that sends
$\mathcal{H}(\D)$ into itself. With this approach, the search of sufficient coefficient conditions 
boils down to finding sufficient conditions for the boundedness of $S_A$.
Therefore, it is not a surprise that many results on slowly growing solutions are inspired by study of 
the classical integral operator $$T_g(f)(z)=\int_0^z f(\z)g'(\z)\,d\z,$$
see \cite{AC:2001, AS:1997, CPPR:preprint, P:1977, SSV:preprint}.
The strength of the operator theoretic approach is demonstrated by proving that the coefficient conditions
arising from Theorem~\ref{th:bloch-general} are
essentially interchangeable with $A\in\L{1}$, see Theorem~\ref{th:interc}.

Deep duality relations are implicit in the proofs of Theorems~\ref{th:bounded},
\ref{th:bloch-general} and  \ref{th:bmoa}.
The dual of $H^1$ is isomorphic to $\rm{BMOA}$ with the Cauchy pairing by the Fefferman
duality relation \cite[Theorem~7.1]{G:2001}, the dual of the disc algebra is
isomorphic to the space of Cauchy transforms with the dual pairing
$\langle f, K\mu \rangle = \int f \, \overline{d\mu}$ \cite[Theorem~4.2.2]{CMR:2006},
and the dual of $A^1_\omega$ is isomorphic to the Bloch space with the dual pairing
$\langle f, g \rangle_{A^2_\omega} = \int_\D f \overline{g} \, \omega\, dm$ \cite[Corollary~7]{PR:2016}.

Finally, we turn to consider coefficient conditions which place solutions of \eqref{eq:de1} in the Hardy spaces.
Our results are inspired by an~open question, which is closely related to the Hardy-Stein-Spencer formula
\begin{equation}
  \label{eq:hssf}
  \nm{f}_{H^p}^p
  =|f(0)|^p
  +\frac{p^2}{2}\int_\D|f(z)|^{p-2}|f'(z)|^2\log\frac{1}{|z|}\,dm(z),
\end{equation}
that holds for $0<p<\infty$ and $f\in\mathcal{H}(\D)$.
For $p=2$,~\eqref{eq:hssf} is the well-known Littlewood-Paley identity, while
the general case follows from~\cite[Theorem~3.1]{H:1994} by integration.


\begin{question} \label{probnorm}
Let $0<p<\infty$. If $f\in\mathcal{H}(\D)$ then is it true that
\begin{equation} \label{eq:hpest}
\nm{f}^p_{H^p} \leq C(p) \int_\D |f(z)|^{p-2} |f''(z)|^2
(1-|z|^2)^3 \, dm(z) + |f(0)|^p + |f'(0)|^p,
\end{equation}
where $C(p)$ is a positive constant such that $C(p)
\to 0^+$ as $p\to 0^+$?
\end{question}

Affirmative answer to this question
would have an immediate application to differential equations, see Section~\ref{sec:applications}.
In the context of differential equations, it suffices to consider Question~\ref{probnorm}
under the additional assumptions that all zeros of $f$ are simple and $f''$ vanishes
at zeros of $f$. Question~\ref{probnorm} has a straightforward solution for a non-trivial class of functions
as it is shown in Section~\ref{sec:zerofree}.

Function $f\in\mathcal{H}(\D)$ is uniformly
locally univalent if there is a constant $0<\delta\leq 1$
such that $f$ is univalent in each pseudo-hyperbolic disc
$\Delta(z,\delta)=\set{w\in\D : |\varphi_z(w)|<\delta}$ for $z\in\D$.
A partial solution to Question~\ref{probnorm} is given by
Theorem~\ref{prop:main}. Here $a\lesssim b$ means that there exists $C>0$ such that $a\leq Cb$. 
Moreover, $a\asymp b$ if and only if $a\lesssim b$ and $a\gtrsim b$.


\begin{theorem} \label{prop:main}
Let $f\in\mathcal{H}(\D)$, and $k\in\N$.
\begin{enumerate}
\item[\rm (i)]
If $0<p\leq 2$, then
\begin{equation} \label{eq:i}
  \nm{f}^p_{H^p} \lesssim \int_{\D} |f(z)|^{p-2} |f^{(k)}(z)|^2 (1- |z|^2)^{2k-1} \, dm(z) + \sum_{j=0}^{k-1} |f^{(j)}(0)|^p.
\end{equation}

\item[\rm (ii)]
If $2\leq p < \infty$, then
\begin{equation} \label{eq:ii}
  \int_{\D} |f(z)|^{p-2} |f^{(k)}(z)|^2 (1- |z|^2)^{2k-1} \, dm(z) + \sum_{j=0}^{k-1} |f^{(j)}(0)|^p \lesssim \nm{f}^p_{H^p}.
\end{equation}

\item[\rm (iii)]
If $0< p < \infty$ and $f$ is uniformly locally univalent, then \eqref{eq:ii} holds.
\end{enumerate}
The comparison constants are independent of $f$; in~\emph{(i)} and~\emph{(ii)} they depend on~$p$,
and in~\emph{(iii)} it depends on $\delta$ and $p$.
\end{theorem}

The proof of Theorem~\ref{prop:main} is presented in Section~\ref{sec:hardy}, and
it takes advantage of a norm in~$H^p$, given in terms of higher derivatives
and area functions, and the boundedness of the non-tangential maximal function.


\section{Zero distribution of solutions} \label{sec:sproof}

For $0\leq p < \infty$, the growth space $H^\infty_p$ consists of those $g\in \mathcal{H}(\D)$ for which
\begin{equation*}
  \nm{g}_{H^\infty_p} = \sup_{z\in\D} \,|g(z)| (1-|z|^2)^p < \infty.
\end{equation*}
We write $H^\infty = H^\infty_0$, for short. The sequence $\{z_n\}_{n=1}^\infty\subset \D$
is called uniformly separated~if
\begin{equation*}
  \inf_{k\in\N} \, \prod_{n\in\N\setminus\{k\}} \left| \frac{z_n - z_k}{1-\overline{z}_n z_k} \right|>0,
\end{equation*}
while $\{z_n\}_{n=1}^\infty\subset \D$ is said to be separated in the hyperbolic metric if
there exists a~constant $\delta>0$ such that
$|z_n-z_k|/|1-\overline{z}_n z_k|>\delta$ for any $n\neq k$.
After the proof of Theorem~\ref{Theorem:SchwarzForHigherOrder1}, we present an auxiliary result
which provides an estimate for the number of sequences in the finite union
appearing in the claim.


\begin{proof}[Proof of Theorem~\ref{Theorem:SchwarzForHigherOrder1}]
(i) If $f$ is a non-trivial solution of \eqref{eq:de3},  then $g=f\circ\vp_a$ solves
\begin{equation} \label{eq:de3g}
  g'''+B_2g''+B_1g'+B_0g=0,
\end{equation}
where
\begin{equation} \label{eq:coeffB}
  \begin{split}
    B_0 & =(A_0 \circ \vp_a)(\vp_a')^3,\qquad B_2=(A_2 \circ \vp_a)\vp_a'-3\, \frac{\vp_a''}{\vp_a'},\\
    B_1 &=(A_1 \circ \vp_a)(\vp_a')^2-(A_2 \circ \vp_a)\vp_a''+ 3 \left(\frac{\vp_a''}{\vp_a'}\right)^2 - \frac{\vp_a'''}{\vp_a'}.
  \end{split}
\end{equation}
By a conformal change of variable, we deduce $\|B_0\|_{H^\infty_3} =  \nm{A_0}_{H^\infty_3}$,
\begin{align*}
  \|B_2\|_{H^\infty_1} & \le \sup_{z\in\D} \, |A_2(z)| \, (1-|z|^2) 
                         + \sup_{z\in\D} \, \frac{6|a|}{|1-\overline{a}z|}\, (1-|z|^2)
                         \leq  \|A_2\|_{H^\infty_1}+12,\\
  \|B_1\|_{H^\infty_2}
                       & \leq \sup_{z\in\D} \, |A_1(z)| \, (1-|z|^2)^2 
                         + \sup_{w\in\D} \, |A_2(w)| \, (1-|w|^2)
                         \left| \frac{\vp_a''(\vp_a(w))}{\vp_a'(\vp_a(w))} \right|(1-|\vp_a(w)|^2)\\
                       & \qquad + \sup_{z\in\D} \, \frac{12|a|^2}{|1-\overline{a}z|^2} \, (1-|z|^2)^2
                         + \sup_{z\in\D} \, \frac{6|a|^2}{|1-\overline{a}z|^2} \,  (1-|z|^2)^2\\
                       & \le\|A_1\|_{H^\infty_2}+4\|A_2\|_{H^\infty_1}+72.
\end{align*}

Let $\mathcal{Z}= \mathcal{Z}(f)$ be the sequence of two-fold zeros of $f$, and
let $a\in\mathcal{Z}$; we may assume that $\mathcal{Z}$ is not empty, for otherwise 
there is nothing to prove. Then, the zero of
$g=f\circ\vp_a$ at the origin is two-fold. By applying Jensen's
formula to $z \mapsto g(z)/z^2$ we obtain
\begin{equation} \label{eq:jensenpreint}
  \sum_{\substack{z_k\in \mathcal{Z} \\ 0<|\vp_a(z_k)|<r}} \log\frac{r}{|\vp_a(z_k)|} \le\frac1{2\pi}\int_0^{2\pi}
  \log^+\left|\frac{g(re^{i\t})}{g''(0)}\right|\,d\t+\log\frac{2}{r^{2}},
  \quad 0<r<1,
\end{equation}
where $\log^+ x = \max\{ 0, \log x \}$ for $0\leq x < \infty$. Since
\begin{align*}
  \int_0^1 \Bigg( \sum_{\substack{z_k\in \mathcal{Z} \\ 0<|\vp_a(z_k)|<r}} \log\frac{r}{|\vp_a(z_k)|} \Bigg) r \, dr
& = \sum_{z_k\in\mathcal{Z}\setminus\{a\}} \int_{|\vp_a(z_k)|}^1 r \log\frac{r}{|\vp_a(z_k)|}\, dr\\
& \geq \frac{1}{8} \sum_{z_k\in\mathcal{Z}\setminus\{a\}} \left( 1 - |\vp_a(z_k)|^2 \right)^2,
\end{align*}
the estimate \eqref{eq:jensenpreint} implies
\begin{equation*}
  \sum_{z_k\in\mathcal{Z}\setminus\{a\}} \left( 1 - |\vp_a(z_k)|^2 \right)^2
  \leq 4 \int_\D\log^+\left|\frac{g(z)}{g''(0)}\right|\,dm(z)+  4\log 2+4.
\end{equation*}

Consider the normalized solution $h(z) = g(z)/g''(0)$ of \eqref{eq:de3g},
which has the initial values $h(0)=h'(0) = 0$ and $h''(0) = 1$.
By the proofs of the growth estimates \cite[Theorems~3.1 and~4.1, and Corollary~4.2]{HKR:2004}, there exists
a constant $C_1>0$ (depending only on the order of the differential equation) such that
\begin{equation*} 
  \frac{1}{2\pi} \int_0^{2\pi} \log^+ \big| h(re^{i\theta}) \big| \, d\theta \le
  C_1\, \sum_{j=0}^{2}\sum_{n=0}^j\int_0^{2\pi}\!\!\int_0^r|B_j^{(n)}(se^{i\t})|(1-s)^{3-j+n-1}\,ds\, d\t.
\end{equation*}

By Cauchy's integral formula and the estimates above, 
there exists a positive constant $C_2 = C_2(\nm{A_0}_{H^\infty_3}, \nm{A_1}_{H^\infty_2}, \nm{A_2}_{H^\infty_1})$ such that
\begin{equation*}
  \bnm{B_j^{(n)}}_{H^\infty_{3-j+n}} \leq C_2,
  \quad j=0,1,2, \quad n=0,\dotsc,j.
\end{equation*}
Let $M_\infty\big(s, B_j^{(n)}\big)$ denote the maximum modulus of $B_j^{(n)}$ on the circle of radius $s$. Now
\begin{equation*} \label{eq:supsup}
  \begin{split}
    & \sup_{a\in\mathcal{Z}} \, \sum_{z_k\in\mathcal{Z}\setminus\{a\}} \!\left( 1 - |\vp_a(z_k)|^2 \right)^2\\
    & \qquad \leq 4 \log 2 + 4 + 16\pi\,  C_1 \, \sup_{a\in\mathcal{Z}} \,\sum_{j=0}^2
    \sum_{n=0}^j \int_0^1 \!\int_0^r M_\infty\big(s, B_j^{(n)}\big)(1-s)^{2-j+n} \, ds\, dr \\
    & \qquad \leq 4 \log 2 + 4 + 16\pi \, C_1 C_2 \,
    \sum_{j=0}^2 \sum_{n=0}^j \int_0^1 \!\int_0^r  \frac{ds}{1-s^2} \, dr < \infty.
  \end{split}
\end{equation*}
The assertion follows from Lemma~\ref{lemma:finmeasure2}(i) below.

(ii) As in the proof of (i), we conclude that $g = f \circ \vp_a$ is a solution of \eqref{eq:de3g},
where the coefficients $B_0,B_1,B_2$ depend on $a\in\D$. By taking advantage of \eqref{eq:assumpp}, 
\begin{equation*} 
  \sup_{a\in\D} \, \int_\D |B_j^{(n)}(z)|(1-|z|^2)^{2-j+n}\, dm(z) < \infty, \quad j=0,\dotsc, 2, \quad n=0,\dotsc, j.
\end{equation*}
First, get rid of the derivatives by standard estimates, and second,
integrate the coefficients \eqref{eq:coeffB} term-by-term.

Let $\mathcal{Z}$ be the sequence of two-fold zeros of $f$. As above, there exists a constant 
$C_3>0$ (depending only on the order of the differential equation) such that
\begin{equation*} 
  \sup_{a\in\mathcal{Z}} \sum_{\substack{z_k\in \mathcal{Z} \\ 0<|\vp_a(z_k)|<r}} \log\frac{r}{|\vp_a(z_k)|}
  \leq \log\frac{2}{r^{2}} +  C_3\, \sup_{a\in\mathcal{Z}} \, \sum_{j=0}^{2}\sum_{n=0}^j\int_\D |B_j^{(n)}(z)|(1-|z|^2)^{2-j+n}\, dm(z)
\end{equation*}
for $0<r<1$. By letting $r\to 1^{-}$, we obtain
\begin{equation*}
  \sup_{a\in\mathcal{Z}} \, \sum_{z_k\in\mathcal{Z} \setminus\{a\}} \!\left( 1 - |\vp_a(z_k)|^2 \right) < \infty.
\end{equation*}
This implies the assertion by Lemma~\ref{lemma:finmeasure2}(ii) below.
\end{proof}

The following lemma gives a concrete upper bound for the number of
sequences in the finite union appearing in the statement of Theorem~\ref{Theorem:SchwarzForHigherOrder1}.


\begin{lemma} \label{lemma:finmeasure2}
Let $\mathcal{Z} = \set{z_k}$ be a sequence of points in $\D$ such that the multiplicity 
of each point is at most $p\in\N$.
\begin{enumerate}

\item[\rm (i)]
  If
  \begin{equation} \label{eq:lemmaesti}
    \sup_{a\in\mathcal{Z}} \, \sum_{z_k\in\mathcal{Z}\setminus\{a\}} \big( 1 - |\vp_{a}(z_k)|^2 \big)^2 \leq M < \infty,
  \end{equation}
  then $\set{z_k}$ can be expressed as a finite union of at most $M+p$ separated sequences.

\item[\rm (ii)]
  If \begin{equation} \label{eq:lemmaestii}
    \sup_{a\in\mathcal{Z}} \, \sum_{z_k\in\mathcal{Z}\setminus\{a\}} \big( 1 - |\vp_{a}(z_k)|^2 \big) \leq M < \infty,
  \end{equation}
  then $\set{z_k}$ can be expressed as a finite union of at most $M+p$ uniformly separated sequences.
\end{enumerate}
\end{lemma}


\begin{proof}
(i) By the proofs of \cite[Theorem~15 and Lemma~16; pp.~69--71]{DS:2004},
\eqref{eq:lemmaesti} implies that $\mathcal{Z}$ is a finite
union of separated sequences; in \eqref{eq:lemmaesti} it suffices to take the supremum with
respect to $\mathcal{Z}$ instead of $\D$. Assume on contrary to the claim, that every partition of
$\mathcal{Z}$ into subsequences is  a~finite union
of at least $M+p+1$ separated sequences. Then, for each $n\in\N$ there exists a~point $z_{k_n}\in\mathcal{Z}$
such that the number of points
\begin{equation*}
  \#\big\{z_k \in\mathcal{Z}: |\vp_{z_k}(z_{k_n})|\leq 2^{-n}\big\}\geq M+p+1.
\end{equation*}
Now
\begin{align*}
  p + M  & \geq p \, + \sum_{z_k \in \mathcal{Z} \setminus\{z_{k_n}\}} \big( 1 - |\vp_{z_k}(z_{k_n})|^2 \big)^2
           \geq \sum_{z_k \in \mathcal{Z}} \big( 1 - |\vp_{z_k}(z_{k_n})|^2 \big)^2\\
         & \geq \# \big\{z_k \in\mathcal{Z}: |\vp_{z_k}(z_{k_n})|\leq 2^{-n}\big\}\cdot  (1-4^{-n})^2
           \geq (M+p+1) (1-4^{-n})^2.
\end{align*}
By letting $n\to\infty$ we arrive to a contradiction.
Hence $\mathcal{Z}$ can be expressed as a finite union of at most $M+p$ separated sequences.

(ii) It is well-known that, if \eqref{eq:lemmaestii} holds then $\mathcal{Z}$ is a finite
union of uniformly separated sequences (again, it suffices to take the supremum with
respect to $\mathcal{Z}$ instead of $\D$). The finite union contains at most $M+p$ 
separated sequences by an argument similar to that above, and each of these separated
sequences is uniformly separated by \eqref{eq:lemmaestii}.
\end{proof}


\begin{example}
If $\{f,g\}$ is a solution base of \eqref{eq:de1}, then $\{f^2,g^2,fg\}$ is a solution base of
\begin{equation} \label{eq:de3ex}
h'''+4A h' + 2A' h = 0.
\end{equation}
Let us apply this property to a classical example \cite[p.~162]{S:1955}
originally due to Hille \cite[p.~552]{H:1949}.
For $\gamma>0$, the differential equation \eqref{eq:de1} with $A(z)=(1+4\gamma^2)/(1-z^2)^2$, $z\in\D$,
admits the solution
\begin{equation*} 
f(z) = \sqrt{1-z^2} \, \sin \left( \gamma \log\frac{1+z}{1-z} \right), \quad z\in\D.
\end{equation*}
The zeros of $f$ are simple and real, and moreover, the hyperbolic distance between two 
consecutive zeros is precisely $\pi/(2\gamma)$. Consequently, \eqref{eq:de3ex} admits the solution $h=f^2$ whose zero-sequence
is a union of two separated sequences. In fact, this sequence is a~union of two uniformly
separated sequences, since all zeros are real \cite[Theorem~9.2]{D:1970}. 
In this case the coefficients of \eqref{eq:de3ex} satisfy both conditions \eqref{eq:assumpp0}
and \eqref{eq:assumpp}.\hfill $\diamond$
\end{example}


\section{Comparison of the coefficient conditions} \label{sec:comparison}

The following result provides us with a comparison of the coefficient conditions.
The reader is invited to compare our findings to those in \cite[Section~5]{CGP:2014}. 
If $A\in\mathcal{H}(\D)$ and 
\begin{equation}
  \label{eq:hpold}
  \sup_{a\in\D}\,
  \int_{\D} |A(z)|^2(1-|z|^2)^2 (1-|\varphi_a(z)|^2) \, dm(z)
\end{equation}
is finite, then we write $A\in\BMOA''$. Note that $A\in \BMOA''$ if and only if
there exists a~function $g=g(A)\in \BMOA$ such that $A=g''$. Correspondingly, if $A\in\mathcal{H}(\D)$ and 
\begin{equation*}
  \nm{A}_{\LMOA''}^2 = \sup_{a\in\D}\,
  \left( \log \frac{e}{1-|a|}  \right)^2
  \int_{\D} |A(z)|^2(1-|z|^2)^2 (1-|\varphi_a(z)|^2)  \, dm(z)<\infty,
\end{equation*}
then $A\in \LMOA''$. As expected, $\LMOA''$ consists of those functions in $\mathcal{H}(\D)$ which can be represented
as the second derivative of a function in $\LMOA$. For more details on $\LMOA$, see \cite{CGP:2014,SZ:1999}.
Finally, part (iv) of Lemma~\ref{lemma:comparison} gives a~sufficient condition for a lacunary series to be in $\LMOA''$.


\begin{lemma} \label{lemma:comparison}
The following assertions hold:
\begin{itemize}
\item[\rm (i)]
  $\L{\alpha_1} \subsetneq \L{\alpha_2} \subsetneq H^\infty_2$ for any 
  $0<\alpha_2<\alpha_1<\infty$;

\item[\rm (ii)]
  ${\LMOA''} \subsetneq \L{1} \subsetneq \L{\alpha} \subsetneq {\BMOA''} \subsetneq H^\infty_2$ 
  for any $1/2<\alpha<1$;

\item[\rm (iii)]
  $ \L{3/2} \subsetneq \LMOA''$, and
  $\LMOA'' \setminus \bigcup_{1<\alpha<\infty} \L{\alpha}$ is non-empty;

\item[\rm (iv)]
  if $\{n_k\}_{k=1}^\infty\subset \N$ and $\{a_k\}_{k=1}^\infty\subset \C$ 
  satisfy the conditions $\inf_{k\in\N} \, n_{k+1}/n_k>1$ and $\sum_{k=1}^\infty |a_k|^2 (\log n_k)^3/n_k^4<\infty$, 
  then $\big( \sum_{k=1}^\infty a_k z^{n_k} \big) \in\LMOA''$.
\end{itemize}
\end{lemma}


\begin{proof}
As (i) is an immediate consequence of the definitions, we proceed to prove (ii).
Let $A\in\LMOA''$. By \eqref{eq:bmoacond2} and  the subharmonicity of $|A|^2$,
we deduce $\nm{A}_{\L{1}}^2 \lesssim  \nm{A}_{\LMOA''}^2$.
Assume on contrary to the assertion
that $\LMOA''=\L{1}$. By \cite[Theorem~1]{GPR:2014}, there exist
$A_0,A_1\in \mathcal{H}(\D)$ satisfying
\begin{equation*}
  |A_0(z)| + |A_1(z)| \asymp \frac{1}{( 1-|z|^2)^2 \log\frac{e}{1-|z|}},
  \quad z\in\D.
\end{equation*}
Since $A_0,A_1\in \LMOA''$, we deduce
\begin{align*}
  \int_{S_a} \frac{dm(z)}{(1-|z|^2)\big( \log\frac{e}{1-|z|} \big)^2}
  \lesssim \int_{S_a} \big( |A_0(z)| + |A_1(z)| \big)^2 (1-|z|^2)^3 \, dm(z)
  \lesssim \frac{1-|a|}{\big( \log\frac{e}{1-|a|} \big)^2}
\end{align*}
as $|a|\to 1^-$. This contradicts the fact
\begin{equation*}
  \int_{S_a} \frac{dm(z)}{(1-|z|^2)\big( \log\frac{e}{1-|z|} \big)^2} \asymp \frac{1-|a|}{ \log\frac{e}{1-|a|}},
  \quad |a|\to 1^-,
\end{equation*}
and hence $\LMOA'' \neq \L{1}$. The remaining part of (ii) is a straightforward computation. Note that 
the inclusion $\L{\alpha} \subsetneq {\BMOA''}$, for any $1/2<\alpha<\infty$, is strict by $A(z)=(1-z)^{-2}$.

To prove (iii) it suffices to prove the latter assertion, as $\L{3/2} \subset \LMOA''$ follows
directly from \eqref{eq:bmoacond2}. If $A(z)=(1-z)^{-2}\big( \log\frac{e}{1-z} \big)^{-1}$ for $z\in\D$, then
$A\notin \bigcup_{1<\alpha<\infty} \L{\alpha}$. To show that $A\in\LMOA''$, it is enough to verify \eqref{eq:bmoacond2}
for $0<a<1$. Since
\begin{equation*}
  \abs{\log\frac{e}{1-z}}
  \geq\log\frac{e}{|1-z|}
  \geq\log\frac{e}{2(1-a)},\quad z\in S_a,
\end{equation*}
we conclude 
\begin{equation*}
  \label{eq:exampleH2log}
  \begin{split}
    & \sup_{0<a<1} \, \frac{\brac{\log\frac{e}{1-a}}^2}{1-a}\int_{S_a}|A(z)|^2(1-|z|^2)^3\,dm(z)\\
    & \qquad \lesssim  \sup_{0<a<1} \, 
    \frac{1}{1-a}\int_a^1\!
    \int_0^{2\pi}\frac{d\theta}{|1-re^{i\theta}|^4} (1-r^2)^3\, r\,dr < \infty.
  \end{split}
\end{equation*}

In order to prove (iv), let $A(z) = \sum_{k=1}^\infty a_k z^{n_k}$ for $z\in\D$.
If $h(z) = \sum_{k=1}^\infty z^{n_k}$ for $z\in\D$, then $h\in\mathcal{B}$
with $M_\infty(r,h) = \sum_{k=1}^\infty r^{n_k} \lesssim \log\frac{e}{1-r}$ for $0<r<1$. 
By the Cauchy-Schwarz inequality,
\begin{equation*}
  M_\infty(r,A) \lesssim \left( \, \sum_{k=1}^\infty |a_k|^2 r^{n_k} \!\right)^{1/2} \left( \log\frac{e}{1-r} \right)^{1/2},
  \quad 0<r<1.
\end{equation*}
It follows that
\begin{align*}
  & \sup_{a\in\D}\,
    \frac{\big( \log \frac{e}{1-|a|}\big)^2}{1-|a|}
    \int_{S_a} |A(z)|^2(1-|z|^2)^3 \, dm(z)\\
  & \qquad \lesssim \int_0^1 M_\infty(r,A)^2 (1-r)^3 \left( \log\frac{e}{1-r} \right)^2 \, dr\\
  & \qquad \lesssim \sum_{k=1}^\infty |a_k|^2 \int_0^1 r^{n_k} (1-r)^3 \left( \log\frac{e}{1-r} \right)^3 \, dr
    \asymp \sum_{k=1}^\infty |a_k|^2 \, \frac{(\log n_k)^3}{n_k^4},
\end{align*}
where the asymptotic equality follows from \cite[Lemma~1.3]{PR:2014}.
This concludes the proof of Lemma~\ref{lemma:comparison}.
\end{proof}


\section{Bounded solutions}\label{sc:bounded}

We consider bounded solutions of~\eqref{eq:de1}. As usual, the space~$H^\infty$ consists of 
$f \in\mathcal{H}(\D)$ for which $\nm{f}_{H^\infty}=\sup_{z\in\D}|f(z)|<\infty$.
The proof of Theorem~\ref{th:bounded} takes advantage of the well-known representation formula
\begin{equation} \label{eq:cauchyIF}
  g(\zeta)=\frac{1}{2\pi}\int_0^{2\pi}\frac{g(e^{it})}{1-e^{-it}\zeta}\,dt,\quad \zeta\in\D,
\end{equation}
which holds for any $g\in H^1$~\cite[Theorem~3.6]{D:1970}.

Let $M$ be the collection of all (finite) complex Borel measures on~$\T = \partial\D$. For $\mu\in M$, 
the total variation measure $|\mu|$ is defined as a set function
\begin{equation*}
  |\mu|(E)=\sup\,\sum_{j}|\mu(E_j)|,
\end{equation*}
where the supremum is taken over all countable partitions $\set{E_j}$ of $E\subset\T$. 
Moreover, $\nm{\mu}=|\mu|(\T)$ is the total variation of $\mu$ \cite[Chapter~6]{R:1987}.
Let~$\mathcal{K}$ be the space of Cauchy transforms, which consists of those analytic functions 
in~$\D$ that are of the form
\begin{equation*}
  (K\mu)(z)=\int_\T\frac{d\mu(\zeta)}{1-\overline{\zeta}z},\quad z\in\D,
\end{equation*}
for some $\mu\in M$. For each $f\in\mathcal{K}$ there is a set $M_f=\big\{\mu\in M : f=K\mu\big\}$
of measures that represent $f$, and produce the norm
\begin{equation*}
  \nm{f}_{\mathcal{K}}=\inf\big\{\nm{\mu} :  \mu\in M_f\big\}.
\end{equation*}
We refer to \cite{CMR:2006} for more details.


\begin{proof}[Proof of Theorem~\ref{th:bounded}]
Let $f$ be any solution of~\eqref{eq:de1}, and write $f_r(z)=f(rz)$ for $0\leq r<1$. Then $f_r$ is
analytic in~$\overline{\D}$ and satisfies $f_r''(w)+r^2A(rw) f_r(w)=0$. By \eqref{eq:diffrep},
\eqref{eq:cauchyIF} for $g=f_r$, and Fubini's theorem, we conclude
\begin{equation*}
  \label{eq:bounded}
  \begin{split}
    f_r(z)
    &=-\frac{1}{2\pi}\int_0^{2\pi} f_r(e^{it})
    \int_0^z\int_0^\zeta\frac{r^2A(rw)}{1-e^{-it}w}\,dw\,d\zeta\,dt+f_r'(0)z+f_r(0), \quad z\in\D.
  \end{split}
\end{equation*}
For all $0<r<1$ sufficiently large, and $z\in\D$, there exists $\mu_{r,z}\in M$ such that
\begin{equation}
  \label{eq:repmeas}
  A_{r,z}(u)=(K\mu_{r,z})(u),\quad u\in\D,
\end{equation}
and $\nm{\mu_{r,z}}<\delta$ for some absolute constant $0<\delta<1$.
Hence, by~\cite[Theorem~4.2.2]{CMR:2006},
\begin{equation*}
  \label{eq:bounded2}
  \begin{split}
    f_r(z)
    &=-\frac{r^2}{2\pi}\int_0^{2\pi}f_r(e^{it})\overline{(K\mu_{r,z})(e^{it})}\,dt
    +f_r'(0)z+f_r(0)\\
    &=- r^2 \int_\T f_r(x)\overline{d\mu_{r,z}(x)}+f_r'(0)z+f_r(0).
  \end{split}
\end{equation*}
By~\cite[Theorem~6.12]{R:1987}, there exist measurable functions $h_{r,z}$ such that $|h_{r,z}(\zeta)|=1$ 
for all  $\zeta\in\T$ and the polar decompositions $d\mu_{r,z}=h_{r,z}\,d|\mu_{r,z}|$ hold. Therefore
\begin{equation*}
  \label{eq:bounded3}
  \begin{split}
    |f_r(z)|
    &\leq\abs{\int_\T f_r(x)\overline{h_{r,z}(x)}\,d|\mu_{r,z}|(x)}+|f_r'(0)|+|f_r(0)|\\
    &\leq\nm{f_r}_{H^\infty} \int_\T \,d|\mu_{r,z}|+|f_r'(0)|+|f_r(0)|\\
    &\leq\nm{f_r}_{H^\infty}\nm{\mu_{r,z}}+|f'(0)|+|f(0)|.
  \end{split}
\end{equation*}
The assertion follows.
\end{proof}

For each $0<r<1$ and $z\in\D$, it is easy to see that
\begin{equation*}
  d\mu_{r,z}(x)
  =\overline{\brac{\int_0^z\int_0^\zeta\frac{A(rw)}{x-w}\,dw\,d\zeta}\frac{dx}{2\pi i}}, \quad x\in\T,
\end{equation*}
is one of the representing measures for which~\eqref{eq:repmeas} holds, and hence $\nm{A_{r,z}}_{\mathcal{K}}\leq\nm{\mu_{r,z}}.$


\section{Solutions of bounded and vanishing mean oscillation}\label{Section:BMOA}\label{sc:bmoa}

The space $\BMOA$ consists of those $f\in\mathcal{H}(\D)$ for which
\begin{equation} \label{eq:bmoadef}
    \nm{f}_{\BMOA}^2=\sup_{a\in\D}\;\nm{f_a}_{H^2}^2<\infty,
\end{equation}
where $f_a(z)=f(\varphi_a(z))-f(a)$ and $\varphi_a(z)=(a-z)/(1-\overline{a}z)$ for $a,z\in\D$. By the Littlewood-Paley identity,
\begin{equation}
\label{eq:bmoanorm}
\begin{split}
\nm{f}_{\BMOA}^2 &\leq4 \, \sup_{a\in\D}\, \int_\D|f'(z)|^2(1-|\varphi_a(z)|^2)\,dm(z)\leq 4\, \nm{f}_{\BMOA}^2,
\end{split}
\end{equation}
see~\cite[pp. 228--230]{G:1981}. Clearly, $\BMOA$ is a subspace of the Bloch space $\mathcal{B}$.

A positive Borel measure $\mu$ on $\D$ is called a Carleson measure, if
\begin{equation*}
  \nm{\mu}_{\textrm{Carleson}}=\sup_{a\in\D}\, \frac{\mu(S_a)}{1-|a|}<\infty.
\end{equation*}
The set $S_a=\set{re^{i\theta}\,:\,|a|<r<1,\,|\theta-\arg(a)|\leq (1-|a|)/2}$ denotes the Carleson square 
with respect to $a\in\D\setminus\set{0}$ while $S_0=\D$. There exists a constant $0<\alpha<\infty$ such that
\begin{equation*} 
  \frac{1}{1-|a|}\leq \alpha\, \frac{1-|a|^2}{|1-\overline{a}z|^2}=\alpha\, |\varphi_a'(z)|,\quad z\in S_a,\quad a\in\D,
\end{equation*}
since $|1-\overline{a}z|\leq |1-|a|^2|+||a|^2-\overline{a}z|\lesssim (1-|a|)$. Consequently,
\begin{equation} \label{eq:sstar}
  \nm{\mu}_{\textrm{Carleson}}
  = \sup_{a\in\D} \, \int_{S_a}\frac{1}{1-|a|}\,d\mu(z)
  \leq\alpha \cdot \sup_{a\in\D} \,\int_\D|\varphi_a'(z)|\,d\mu(z).
\end{equation}

We prove Theorem~\ref{th:bmoa2} and consider its counterpart for $\VMOA$.
Theorem~\ref{th:bmoa2} is inspired by \cite[Theorem~3.1]{SZ:1999}.
We return to consider $\BMOA$ and $\VMOA$ solutions in Section~\ref{sec:bmoa_alt},
where parallel results are obtained by using the representation formula for $H^1$ functions.


\begin{proof}[Proof of Theorem~\ref{th:bmoa2}]
The proof consists of two steps. First, we show that
\begin{equation} \label{eq:firstclaim}
  \sup_{1/2<r<1}\, \sup_{a\in\D}\,
  \left( \log \frac{e}{1-|a|}  \right)^2
  \int_{\D} |A(rz)|^2(1-|z|^2)^2 (1-|\varphi_a(z)|^2)  \, dm(z)
  \lesssim \nm{A}_{\LMOA''}^2.
\end{equation}
Denote
\begin{equation*}
  I(a,r)  = \int_{\D} |A(rz)|^2(1-|z|^2)^2 (1-|\varphi_a(z)|^2)  \, dm(z), \quad 0<r<1, \quad a\in\D,
\end{equation*}
for short. Let $1/2<|a|<1/(2-r)$. Since $|1-\overline{a}z| \leq 2 \, |1 - \overline{a}z/r|$ for $|z|\leq r$, 
\begin{align*}
  I(a,r) & = \int_{D(0,r)} |A(z)|^2\big(1-\left| \tfrac{z}{r} \right|^2\big)^3 
           \frac{1-|a|^2}{\left|1-\overline{a} \, \frac{z}{r}\right|^2} \, \frac{dm(z)}{r^2}\\
         &\leq \frac{4}{r^2} \, \int_{\D} |A(z)|^2 (1-|z|^2)^2 (1-|\varphi_a(z)|^2) \, dm(z)
\end{align*}
is uniformly bounded for $1/2<r<1$ and $1/2<|a|<1/(2-r)$. Let $1/(2-r)\leq |a|<1$. Now
\begin{align*}
  I(a,r) & \leq \nm{A}_{\L{1}}^2 \int_{\D} \frac{\big(1-|z|^2\big)^2 \big(1-|\varphi_a(z)|^2\big)}{\big(1-|rz|^2\big)^4 
           \big(\log\frac{e}{1-|rz|} \big)^2} \, dm(z)\\
         & \lesssim \nm{A}_{\L{1}}^2 \int_0^1 \frac{(1-s)^3 (1-|a|)}{(1-rs)^4 
           \big(\log\frac{e}{1-rs} \big)^2 (1-|a|s)} \, ds.
\end{align*}
As $t \mapsto (1-t)^2\big( \log\frac{e}{1-t} \big)$ is decreasing for $0<t<1$, we apply $r\leq 2 - 1/|a|$ to obtain
\begin{align*}
  I(a,r) & \lesssim \nm{A}_{\L{1}}^2 (1-|a|) \int_0^{|a|} \frac{ds}{(1-s)^2 
           \big(\log\frac{e}{1-s} \big)^2}
           + \frac{\nm{A}_{\L{1}}^2}{(1-|a|)^4 
           \big(\log\frac{e}{1-|a|} \big)^2} \int_{|a|}^1 (1-s)^3 \, ds\\
         & \lesssim \nm{A}_{\L{1}}^2 \left( \log\frac{e}{1-|a|} \right)^{-2}
\end{align*}
for all $1/2<r<1$ and $1/(2-r)\leq |a|<1$. 
Since $\nm{A}_{\L{1}}^2 \lesssim \nm{A}_{\LMOA''}^2$ by the proof of Lemma~\ref{lemma:comparison}(ii), 
this completes the proof of \eqref{eq:firstclaim}.

Second, we proceed to consider the differential equation \eqref{eq:de1}.
Let $f$ be a non-trivial solution of \eqref{eq:de1}. 
By Lemma~\ref{lemma:comparison}(ii) and \cite[Corollary~4(b)]{HKR:2016},
we may assume that $f\in\B$. Now, \eqref{eq:de1} and \eqref{eq:bmoanorm} yield
\begin{align*}
  \nm{f_r}_{\BMOA}^2 & \asymp  \sup_{a\in\D} \left( |f'(ra)|^2(1-|a|^2)^2 \, r^2 
                       + \, \int_{\D} r^4|f''(rz)|^2 (1-|z|^2)^2 (1-|\varphi_a(z)|^2) \, dm(z) \right) \\
                     & \lesssim  \nm{f_r}_\B^2 + \sup_{a\in\D} \, \int_{\D} |f_r(z)-f_r(a)|^2 \, |A(rz)|^2 (1-|z|^2)^2 (1-|\varphi_a(z)|^2) \, dm(z)\\
                     & \qquad    + \sup_{a\in\D} \, |f_r(a)|^2 \int_{\D} |A(rz)|^2 (1-|z|^2)^2 (1-|\varphi_a(z)|^2) \, dm(z)\\
                     & \lesssim \nm{f_r}_\B^2 + I_1 + I_2
\end{align*}
with absolute comparison constants. By Carleson's theorem~\cite[Theorem~9.3]{D:1970} and \eqref{eq:bmoadef},
\begin{align*}
  I_1 & \lesssim 
        \sup_{a\in\D} \, \int_{\D} |(f_r)_a(z)|^2 \, \big|A(r \varphi_a(z))\big|^2 \big(1-|\varphi_a(z)|^2\big)^3 \, |\varphi_a'(z)| \, dm(z)\\
      & \lesssim \sup_{a\in\D} \left( \bnm{(f_r)_a}_{H^2}^2 \cdot \sup_{b\in\D} \,
        \int_{\D} \big|A(r \varphi_a(z))\big|^2 \big(1-|\varphi_a(z)|^2\big)^3 |\varphi_a'(z)| |\varphi_b'(z)|\, dm(z) \right)\\
      & \lesssim \nm{f_r}_{\BMOA}^2  \cdot \sup_{c\in\D} \,
        \int_{\D} |A(rz)|^2 (1-|z|^2)^2 \big(1-|\varphi_c(z)|^2\big)\, dm(z).
\end{align*}
Estimation of $I_2$ is easier. By \cite[Corollary~5.3]{G:2001},
\begin{equation*}
  I_2 \lesssim \nm{f_r}_{\BMOA}^2 \cdot \sup_{a\in\D} \,
  \left( \log\frac{e}{1-|a|} \right)^2 \int_{\D} |A(rz)|^2 (1-|z|^2)^2 (1-|\varphi_a(z)|^2) \, dm(z).
\end{equation*}
If~\eqref{eq:bmoacondnew} is sufficiently small, then \eqref{eq:firstclaim} implies that
$\nm{f_r}_{\BMOA}$ is uniformly bounded for $1/2<r<1$. By letting $r\to 1^-$, we conclude $f\in\BMOA$.
\end{proof}

The space $\VMOA$ consists of those $f\in H^2$ for which
\begin{equation*}
  \lim_{|a|\to 1^-} \, \nm{f_a}_{H^2}^2=0,
\end{equation*}
where $f_a$ is the auxiliary function in the beginning of Section~\ref{sc:bmoa}. Clearly, 
$\VMOA$ is a~subspace of the little Bloch space $\mathcal{B}_0$. As Theorem~\ref{th:bmoa2} is motivated by
\cite[Theorem~3.1]{SZ:1999},  the counterpart of the following result is  \cite[Theorem~3.6]{SZ:1999}.


\begin{theorem}
\label{th:vmoa_new}
Let $A\in\mathcal{H}(\D)$. If \eqref{eq:bmoacondnew} is sufficiently small and
\begin{equation*}
  \lim_{|a|\to 1^-}\,
  \left( \log \frac{e}{1-|a|}  \right)^2
  \int_{\D} |A(z)|^2(1-|z|^2)^2 (1-|\varphi_a(z)|^2)  \, dm(z) = 0,
\end{equation*}
then all solutions $f$ of~\eqref{eq:de1} satisfy $f\in\VMOA$.
\end{theorem}

The proof of Theorem~\ref{th:vmoa_new}
is omitted, since it is similar to the proof of Theorem~\ref{th:bmoa2}.


\section{Solutions in the Bloch and the little Bloch spaces}\label{Section:Bloch}

An integrable function $\om:\D\to[0,\infty)$ is called a weight. It is radial if $\om(u)=\om(|u|)$ for all $u\in\D$.
For $0<p<\infty$ and a weight $\om$, the weighted Bergman space $A^p_\om$ consists
of those $f\in\HH(\D)$ for which
\begin{equation*}
  \|f\|_{A^p_\om}^p=\int_\D|f(u)|^p\om(u)\,dm(u)<\infty.
\end{equation*}
For a radial weight $\om$,  we define $\widehat{\om}(u)=\int_{|u|}^1\om(r)\,dr$ for $u\in\D$.
We denote $\om \in \RR$ whenever $\om$ is radial and there exist constants
$C=C(\om)>0$, $\alpha=\alpha(\om)>0$ and $\b=\b(\om)\ge \alpha$ such that
\begin{equation}
  \begin{split}\label{eq:regular}
    C^{-1}\left(\frac{1-r}{1-t}\right)^{\alpha}\widehat{\om}(t)\le\widehat{\om}(r)
    \le C\left(\frac{1-r}{1-t}\right)^{\b}\widehat{\om}(t),\quad 0\le r\le
    t<1.
  \end{split}
\end{equation}

Let $0<p<\infty$ and $\omega$ be a radial weight. If $\widehat{\omega}(r)=0$ for some $0<r<1$, 
then $A^p_\omega=\mathcal{H}(\D)$. Let $\omega$ be  a radial weight such that $\widehat{\omega}(r)>0$ 
for all $0\leq r<1$. By standard estimates,
\begin{equation*}
  \nm{f}_{A^p_\omega}^p \gtrsim M_p(r,f)^p \, \widehat{\omega}(r)
  \gtrsim M_\infty(2r-1,f)^p(1-r)
  \, \widehat{\omega}(r), \quad 1/2<r<1,
\end{equation*}
where $M_p(r,f)$ denotes the $H^p$ mean of $f$, and hence
\begin{equation} \label{eq:converg}
  |f(z)| \lesssim \frac{\nm{f}_{A^p_\omega}}{\widehat\omega(z)^{1/p}(1-|z|)^{1/p}}, \quad 1/2<|z|<1.
\end{equation}
We will concentrate on the case $p=2$. By \eqref{eq:converg},
the norm convergence in $A^2_\om$ implies the uniform convergence on compact subsets of $\D$,
and consequently each point evaluation $L_\z(f)=f(\z)$ is a~bounded linear functional in the Hilbert space~$A^2_\om$. 
Hence, there exist unique reproducing kernels $B^\om_\z\in A^2_\om$ with
$\|L_\z\|=\|B^\om_\z\|_{A^2_\om}$ such that
\begin{equation}\label{eq:Bergmanrep-general}
  f(\z)=\langle f,B^\om_\z\rangle_{A^2_\om}=\int_\D f(u)\overline{B^\om_\z(u)}\om(u)\,dm(u),\quad f\in A^2_\om,
\end{equation}
Moreover, the normalized monomials $(2 \, \om_{2n+1})^{-1/2}\, z^n$, for $n\in\N \cup \{0\}$,
form the standard orthonormal basis of $A^2_\om$ and hence
\begin{equation} \label{repker}
  B^\om_\z(u)=\sum_{n=0}^\infty\frac{(u \overline{\z})^n}{2\, \om_{2n+1}},\quad u,\z\in\D;
\end{equation}
see \cite[Theorem~4.19]{Z:2007} for details in the classical case.
Here $\om_x=\int_0^1r^x\om(r)\,dr$ for $1\leq x < \infty$.
Weight $\om$ is called normalized if $\om_1=1/2$, which implies that $\om(\D)=\int_\D\om(u)\,dm(u)=2\om_1=1$.

We begin with a lemma which shows that the derivative of $B_\z^\om$ is closely related to
the reproducing kernel of another Bergman space with a suitable chosen weight.
For example, $B^\omega_\z(u) = (1-u\overline{\zeta})^{-2-\alpha}$ is the reproducing kernel
corresponding to the standard weight $\omega(u) = (\alpha+1)(1-|u|^2)^\alpha$, $\alpha>-1$, while
$(B_\z^\omega)'(u) = (2+\alpha)\overline{\z}(1-u \overline{\z})^{-3-\alpha}$ is related to the reproducing 
kernel of the Bergman space with the weight $\widetilde{\om}(u) = (1-|u|^2)^{\alpha+1}$. In general, we define
\begin{equation*}
  \widetilde{\om}(u)  =2\int_{|u|}^1\om(r)r\,dr, \quad u\in\D,
\end{equation*}
for any radial weight $\om$.


\begin{lemma}\label{Lemma:kernels-derivative}
If $\om$ is radial then $(B_\z^\om)'(u)=\overline{\z} \, B_\z^{\widetilde{\om}}(u)$ for $u,\zeta\in\D$.
\end{lemma}


\begin{proof}
It is clear that representations \eqref{repker}
exist for both $B_\zeta^{\omega}$ and $B_\zeta^{\widetilde{\omega}}$. By Fubini's theorem,
\begin{equation*}
  \widetilde{\om}_{2n+1}
  =2\int_0^1 \om(s) s\int_0^sr^{2n+1}\,dr\,ds=\frac{\om_{2n+3}}{n+1},\quad n\in\N\cup\{0\},
\end{equation*}
and hence
\begin{equation*}
  (B_\z^\om)'(u)=\overline{\z}\,
  \sum_{n=0}^\infty\frac{(n+1)(u \overline{\z})^n}{2\, \om_{2n+3}}=\overline{\z} \, B_\z^{\widetilde{\om}}(u),\quad u,\zeta\in\D.
\end{equation*}
This proves the assertion.
\end{proof}

The following auxiliary result is well-known to experts. For a radial weight $\om$, we define
\begin{equation*}
  \om^\star(u)  =\int_{|u|}^1\log\frac{r}{|u|}\, \om(r)\,r \, dr,\quad u\in\D\setminus\{0\}.
\end{equation*}


\begin{lemma}\label{lm:green2.1}
If $f,g\in H^2$, then
\begin{equation}\label{eq:greentd}
  \frac{1}{2\pi}\int_0^{2\pi} f(e^{it})\overline{g(e^{it})}\,dt
  =2\int_\D f'(u)\overline{g'(u)}\log\frac{1}{|u|}\,dm(u)+f(0)\overline{g(0)}.
\end{equation}
Moreover, if $f,g\in\HH(\D)$ and $\om$ is a normalized radial weight, then
\begin{equation*}
  \langle f,g\rangle_{A^2_\om}=4 \, \langle f',g'\rangle_{A^2_{\om^\star}}+f(0)\overline{g(0)}.
\end{equation*}
\end{lemma}


\begin{proof}
Identity \eqref{eq:greentd} is a special case of~\cite[Theorem~9.9]{Z:2007}.
Let $f,g\in\HH(\D)$. By \eqref{eq:greentd},
\begin{equation*}
  \begin{split}
    \frac1{\pi}\int_0^{2\pi} f(re^{it})\overline{g(re^{it})}\,dt
    &=4\int_{D(0,r)} f'(u)\overline{g'(u)}\log\frac{r}{|u|}\,dm(u) + 2 f(0)\overline{g(0)}.
  \end{split}
\end{equation*}
The assertion follows by integrating both sides with respect to the measure $\omega(r)r\,dr$ and using Fubini's theorem.
\end{proof}

Recall that the Bloch space $\mathcal{B}$ consists of those $f\in\HH(\D)$ for which
\begin{equation*}
  \nm{f}_{\mathcal{B}}=\sup_{z\in\D}\, |f'(z)| (1-|z|^2) <\infty.
\end{equation*}


\begin{theorem}\label{th:bloch-general}
Let $\om\in\RR$ be normalized, and $A\in\HH(\D)$ such that
\begin{equation} \label{eq:Elimsup}
  \limsup_{r\to 1^-}\;\sup_{z\in\D}\,(1-|z|^2)\int_\D\left|\int_0^z\overline{(B_\z^\om)'(u)} A(r\z)\,d\z\right|
  \frac{\om^\star(u)}{1-|u|^2}\,dm(u) <\frac14.
\end{equation}
Then every solution $f$ of~\eqref{eq:de1} satisfies $f\in\mathcal{B}$, and
\begin{equation*}
  \|f\|_\B\le\frac{1}{1-4 \, X_\B(A)}\left(|f(0)|\, \sup_{z\in\D}\, (1-|z|^2)\left|\int_0^zA(\z) \, d\z\right|+|f'(0)|\right),
\end{equation*}
where
\begin{equation*}
  X_\B(A)=\sup_{z\in\D}\,(1-|z|^2)\int_\D\left|\int_0^z\overline{(B_\z^\om)'(u)}
    A(\z)\,d\z\right|\frac{\om^\star(u)}{1-|u|^2}\,dm(u) < \frac{1}{4}.
\end{equation*}
\end{theorem}


\begin{proof}
Observe that $\om^\star(u)/(1-|u|^2)\asymp\widetilde{\om}(u)$ as $|u|\to1^-$,
since $\om\in\RR$ by the hypothesis.
For fixed $z\in\D$, Fubini's theorem and Lemma~\ref{Lemma:kernels-derivative} yield
\begin{equation} \label{eq:bff}
  \begin{split}
    &\limsup_{r\to 1^-}\,(1-|z|^2)\int_\D\left|\int_0^z\overline{(B_\z^\om)'(u)}A(r\z)\,d\z\right|\frac{\om^\star(u)}{1-|u|^2}\,dm(u)\\
    & \qquad \gtrsim(1-|z|^2)\int_\D\left|\int_0^z\overline{(B_\z^\om)'(u)}A(\z)\,d\z\right|\widetilde{\om}(u)\,dm(u)\\
    & \qquad \ge (1-|z|^2)\left|\int_0^z \langle 1,B_\z^{\widetilde{\om}}\rangle_{A^2_{\widetilde{\om}}} \, A(\z)\z\,d\z\right|
    \ge(1-|z|^2) \left|\int_0^zA(\z)\z\,d\z\right|,
  \end{split}
\end{equation}
and it follows that $A\in H^\infty_2$. Note that the use of the reproducing formula could be avoided by 
a straightforward integration.

Let $f$ be any solution of~\eqref{eq:de1}, and denote $f_r(z)=f(rz)$ for $0\leq r<1$. Then,
\begin{equation*}
  f_r'(z)=-\int_0^zf_r(\zeta)r^2A(r\zeta)\,d\zeta+f_r'(0),\quad z\in\D.
\end{equation*}
The reproducing formula~\eqref{eq:Bergmanrep-general} and Fubini's theorem imply
\begin{equation*}\label{eq:bloch1}
  \begin{split}
    f_r'(z)
    &=-\int_0^z\left(\int_\D f_r(u)\overline{B_\z^\om(u)}\om(u)\,dm(u)\right)r^2A(r\zeta)\,d\zeta+f_r'(0)\\
    &=-\int_\D f_r(u)\left(\int_0^z\overline{B_\z^\om(u)}r^2A(r\zeta)\,d\zeta\right)\om(u)\,dm(u)
    +f_r'(0), \quad z\in\D,
  \end{split}
\end{equation*}
from which the second part of Lemma~\ref{lm:green2.1} yields
\begin{equation*}
  \label{eq:bloch1.5}
  \begin{split}
    f_r'(z)
    &=-4\int_\D f'_r(u)\left(\int_0^z\overline{(B_\z^\om)'(u)}r^2A(r\zeta)\,d\zeta\right)\om^\star(u)\,dm(u)\\
    &\qquad -f_r(0)\int_0^zr^2A(r\z)\, d\z +f_r'(0), \quad z\in\D.
  \end{split}
\end{equation*}
It follows that
\begin{equation*}\label{eq:bloch3}
  \begin{split}
    \nm{f_r}_{\mathcal{B}}
    &\le4\, \|f_r\|_\B \, \sup_{z\in\D}\, (1-|z|^2)\int_\D\left|\int_0^z\overline{(B_\z^\om)'(u)}
      A(r\zeta)\,d\zeta\right|\frac{\om^\star(u)}{1-|u|^2}\,dm(u)\\
    &\qquad +|f(0)|\, \sup_{z\in\D}\, (1-|z|^2)\left|\int_0^z A(r\z)\, d\z\right|+|f'(0)|,
    \quad 0<r<1.
  \end{split}
\end{equation*}
We deduce $f\in\B$ by re-organizing the terms and letting $r\to 1^-$.
Now that $f\in\B \subset A^2_{\omega}$ (for the inclusion, see \cite[Proposition~6.1]{PelSum14}),
we may repeat the proof from the beginning with $r=1$ to deduce the second part of the assertion.
\end{proof}


\begin{remark} \label{remark:annulus}
The proof of Theorem~\ref{th:bloch-general} shows that, in order to conclude $f\in\mathcal{B}$, it suffices to take the
supremum in \eqref{eq:Elimsup} over any annulus $R<|z|<1$ instead of $\D$.
\end{remark}

We apply an operator theoretic argument to study the sharpness of Theorem~\ref{th:bloch-general}. Let
\begin{equation*}
  I(A,\omega) = \limsup_{r\to 1^-}\;\sup_{z\in\D}\,(1-|z|^2)\int_\D\left|\int_0^z\overline{(B_\z^\om)'(u)} A(r\z)\,d\z\right|
  \frac{\om^\star(u)}{1-|u|^2}\,dm(u)
\end{equation*}
denote the left-hand side of \eqref{eq:Elimsup}, for short.


\begin{theorem} \label{th:interc}
Let $\omega\in\mathcal{R}$ be normalized, and $A\in\mathcal{H}(\D)$.
The following conditions are equivalent:
\begin{enumerate}
\item[\rm (i)] $A\in\L{1}$;
\item[\rm (ii)] $I(A,\omega)<\infty$;
\item[\rm (iii)] the operator $S_A : \mathcal{B} \to \mathcal{B}$ is bounded.
\end{enumerate}
\end{theorem}


\begin{proof}
(i)$\implies$(ii): Observe that $\om^\star(u)/(1-|u|^2)\asymp\widehat{\om}(u)$ as $|u|\to1^-$.
By Fubini's theorem,
\begin{equation*}
  I(A,\omega)  \lesssim \limsup_{r\to 1^-}\;\sup_{z\in\D}\,(1-|z|^2) \int_0^z |A(r\z)|
  \int_\D\big| (B_\z^\om)'(u) \big| \,  \widehat{\om}(u) \,dm(u) \, |d\zeta|,
\end{equation*}
where
\begin{equation*}
  \int_\D\big| (B_\z^\om)'(u) \big| \,  \widehat{\om}(u) \,dm(u)
  \lesssim \int_0^{|\z|} \frac{\widehat{\widehat{\omega}}(t)\, dt}{\widehat{\omega}(t)(1-t)^2}
  \asymp \int_0^{|\z|} \frac{dt}{1-t^2} = \frac{1}{2} \log\frac{1+|\zeta|}{1-|\zeta|}, \quad \zeta\in\D,
\end{equation*}
by \cite[Theorem~1]{PR:2016}, Fubini's theorem and \eqref{eq:regular}. It follows that
$I(A,\omega) \lesssim \nm{A}_{\L{1}}<\infty$.

(ii)$\implies$(iii): This implication follows by an argument similar to the proof
of Theorem~\ref{th:bloch-general}. As in \eqref{eq:bff}, we know that
\begin{equation*}
  \sup_{z\in\D} \, (1-|z|^2)\int_\D\left|\int_0^z\overline{(B_\z^\om)'(u)}A(\z)\,d\z\right|\frac{\om^\star(u)}{1-|u|^2}\,dm(u)
  \leq I(A,\omega) < \infty,
\end{equation*}
and further, $A\in H^\infty_2$. Let $f\in\mathcal{B}\subset A^2_\omega$ (for the inclusion, see 
\cite[Proposition~6.1]{PelSum14}). The reproducing formula \eqref{eq:Bergmanrep-general}, Fubini's theorem
and Lemma~\ref{lm:green2.1} imply
\begin{equation*}
  \bnm{S_A(f)}_{\mathcal{B}} = \sup_{z\in\D} \, (1-|z|^2) \left| \int_0^z f(\zeta) A(\zeta) \, d\zeta \right|
  \lesssim \nm{f}_{\mathcal{B}} \, I(A,\omega) + |f(0)| \cdot \nm{A}_{H^\infty_2},
\end{equation*}
and hence we deduce (iii).

(iii)$\implies$(i): By assumption, there exists a constant $C>0$ such that
\begin{equation} \label{eq:mjm}
  \sup_{z\in\D}  \, (1-|z|^2)^2|f(z)|\, |A(z)| = \bnm{S_A(f)''}_{H^\infty_2}
  \lesssim \bnm{S_A(f)}_{\mathcal{B}} \leq C \big( \nm{f}_{\mathcal{B}} + |f(0)| \big)
\end{equation}
for any $f\in\mathcal{B}$. Consider the family of test functions
\begin{equation*}
  f_{\zeta}(z) = \log\frac{e}{1-\overline{\zeta} z}, \quad z,\zeta\in\D,
\end{equation*}
for which $\sup_{\zeta\in\D} \, \nm{f_\zeta}_{\mathcal{B}} \leq 2$. By \eqref{eq:mjm},
\begin{equation*}
  (1-|z|^2)^2 \left|\log\frac{e}{1-\overline{\zeta}z} \right| |A(z)| \leq 3C, \quad z,\zeta\in\D,
\end{equation*}
which gives the condition (i) for $\z=z$.
\end{proof}

If
\begin{equation}
  \label{eq:Elimsupnor}
  \sup_{z\in\D}\, (1-|z|^2)\int_\D\left|\int_0^z\overline{(B_\z^\om)'(u)} A(\z)\,d\z\right|
  \frac{\om^\star(u)}{1-|u|^2}\,dm(u)
\end{equation}
is sufficiently small, then a close look at the proof of Theorem~\ref{th:interc} implies that
\eqref{eq:Elimsup} is satisfied. As a~consequence, we obtain the following result.


\begin{corollary}\label{cor:bloch-general}
Let $\om\in\RR$ be normalized, and $A\in\HH(\D)$ such that \eqref{eq:Elimsupnor} is sufficiently small.
Then every solution $f$ of~\eqref{eq:de1} satisfies $f\in\mathcal{B}$.
\end{corollary}

The little Bloch space $\mathcal{B}_0$ consists of those $f\in\mathcal{H}(\D)$ for which
\begin{equation*}
  \lim_{|z|\to 1^-} |f'(z)| (1-|z|^2)=0.
\end{equation*}
The following result is a counterpart of Theorem~\ref{th:bloch-general} concerning the little Bloch space.


\begin{theorem}
\label{b0short}
Let $\om\in\RR$ be normalized, and $A\in\HH(\D)$ such that
\begin{equation*}
  \limsup_{|z|\to 1^{-}}\, (1-|z|^2)
  \int_{\D}\abs{\int_0^z \overline{(B_\zeta^\omega)'(u)} A(\zeta)\,d\zeta}
  \frac{\omega^\star(u)}{1-|u|^2}\,dm(u)=0.
\end{equation*}
Then every solution $f$ of~\eqref{eq:de1} satisfies $f\in\mathcal{B}_0$.
\end{theorem}


\begin{proof}
As in \eqref{eq:bff}, we conclude
\begin{equation*}
  \limsup_{|z|\to 1^{-}} \, (1-|z|^2) \left| \int_0^z A(\zeta) \zeta \, d\zeta\right| = 0.
\end{equation*}

By assumption, there exists a constant $0<R<1$ such that
\begin{equation*}
  \sup_{R< |z| < 1} \, (1-|z|^2) \int_{\D}\abs{\int_0^z \overline{(B_\zeta^\omega)'(u)} A(\zeta)\,d\zeta}
  \frac{\omega^\star(u)}{1-|u|^2}\,dm(u) < \frac{1}{8}.
\end{equation*}
For fixed $z$, $R<|z|<1$, Lebesgue's dominated convergence theorem implies
\begin{equation*}
  \limsup_{r\to 1^-}\,(1-|z|^2)\int_\D\left|\int_0^z\overline{(B_\z^\om)'(u)}A(r\z)\,d\z\right|\frac{\om^\star(u)}{1-|u|^2}\,dm(u)
  <\frac{1}{8}.
\end{equation*}
We deduce a counterpart of \eqref{eq:Elimsup} with the supremum taken over the annulus $R< |z| < 1$.
By Remark~\ref{remark:annulus}, it follows that any solution $f$ of \eqref{eq:de1}
satisfies $f\in \mathcal{B} \subset A^2_\omega$ (for the inclusion, see \cite[Proposition~6.1]{PelSum14}).
As in the proof of Theorem~\ref{th:bloch-general}, we have
\begin{equation*}
  \begin{split}
    (1-|z|^2) |f'(z)|
    &\le4\, \|f\|_\B \, (1-|z|^2) \int_\D\left|\int_0^z\overline{(B_\z^\om)'(u)} A(\zeta)\,d\zeta\right|\frac{\om^\star(u)}{1-|u|^2}\,dm(u)\\
    &\quad +|f(0)| \, (1-|z|^2)\left|\int_0^z A(\z)\, d\z\right|+(1-|z|^2)|f'(0)|,
    \quad R < |z| < 1.
  \end{split}
\end{equation*}
The assertion follows.
\end{proof}

If $A\in\mathcal{H}(\D)$ and
\begin{equation*}
  \lim_{|z|\to 1^-} |A(z)| (1-|z|^2)^2\log\frac{e}{1-|z|} =0,
\end{equation*}
then every solution $f$ of~\eqref{eq:de1} satisfies $f\in\mathcal{B}_0$.
Actually, a straightforward modification of the proof of Corollary~\ref{cor:bloch-general}, 
by taking account on Remark~\ref{remark:annulus}, implies that $f\in\mathcal{B}$. Therefore
\begin{equation*}
  f''(z)=-A(z) \, \int_\D \frac{f(u)}{(1-\overline{u}z)^2}\,dm(u),\quad z\in\D.
\end{equation*}
By applying Lemma~\ref{lm:green2.1} twice, we obtain
\begin{equation*}
  |f''(z)| \lesssim
  |A(z)|\brac{|f(0)|+|f'(0)|+\nm{f''}_{H^\infty_2}\int_\D\frac{(1-|u|^2)^2}{\abs{1-\overline{u}z}^4}\,dm(u)},
  \quad z\in\D.
\end{equation*}
Since $f\in\mathcal{B}$, we deduce $f''\in H^\infty_2$, and hence the argument above shows that $f\in\mathcal{B}_0$
by~\cite[Lemma~3.10 and Theorem~5.13]{Z:2007}.


\section{Solutions of bounded and vanishing mean oscillation --- parallel results}\label{sec:bmoa_alt}

In this section, we consider two coefficient estimates, which are derived from the representation~\eqref{eq:cauchyIF}.
These estimates give sufficient conditions for all solutions of \eqref{eq:de1} to be in $\BMOA$ or $\VMOA$.
Recall that, by~\eqref{eq:bmoanorm} and \eqref{eq:sstar}, the particular measure 
$d\mu_f(z)=|f'(z)|^2(1-|z|^2)\,dm(z)$ satisfies
\begin{equation}
  \label{eq:cnbmoa}
  \nm{\mu_f}_{\textrm{Carleson}}\lesssim\nm{f}_{\BMOA}^2.
\end{equation}
Actually, $f\in\BMOA$ if and only if $\mu_f$ is a Carleson measure~\cite[Lemma~3.3]{G:1981}.


\begin{theorem} \label{th:bmoa}
Let $A\in\mathcal{H}(\D)$. If
\begin{equation}
  \label{eq:bmoacond}
  \limsup_{r\to 1^-}\,\sup_{a\in\D}\,
  \int_\D\brac{\frac{1}{2\pi}\int_0^{2\pi}
    \abs{\int_0^z\frac{A(r\zeta)\,d\zeta}{1-e^{-it}\zeta}}
    dt}^2(1-|\varphi_a(z)|^2)\,dm(z)
\end{equation}
is sufficiently small, then all solutions $f$ of~\eqref{eq:de1} satisfy $f\in\BMOA$.
\end{theorem}


\begin{proof}
By applying \eqref{eq:cauchyIF}  to $g\equiv 1$, we obtain
\begin{equation} \label{eq:Aprim2}
  \begin{split}
    \abs{\int_0^zA(r\zeta)\,d\zeta}
    &=\abs{\frac{1}{2\pi}\int_0^{2\pi}\int_0^z\frac{A(r\zeta)\,d\zeta}{1-e^{-it}\zeta}\,dt}
    \leq\frac{1}{2\pi}\int_0^{2\pi}\abs{\int_0^{z}\frac{A(r\zeta)\,d\zeta}{1-e^{-it}\zeta}}\,dt,
  \end{split}
\end{equation}
for $0\leq r\leq1$ and $z\in\D$. By~\eqref{eq:bmoanorm} and~\eqref{eq:bmoacond},
any second primitive of $A$ belongs to $\BMOA$.

Let $f$ be a solution of~\eqref{eq:de1}, and denote $f_r(z)=f(rz)$ for $0\leq r<1$. Then $f_r$ is
analytic in~$\overline{\D}$ and satisfies $f_r''(\zeta)+r^2A(r\zeta)f_r(\zeta)=0$. We deduce
\begin{equation*}
  f_r'(z)=-\int_0^zf_r(\zeta)r^2A(r\zeta)\,d\zeta+f_r'(0),\quad z\in\D.
\end{equation*}
By~\eqref{eq:cauchyIF} and Fubini's theorem,
\begin{equation*} \label{eq:bmoa}
  \begin{split}
    f_r'(z)
    &=-\frac{1}{2\pi}\int_0^{2\pi} f_r(e^{it})
    \int_0^z\frac{r^2A(r\zeta)}{1-e^{-it}\zeta}\,d\zeta\,dt+f_r'(0)\\
    &=-\frac{r^2}{2\pi}\int_0^{2\pi} f_r(e^{it})\overline{g_{r,z}(e^{it})}\,dt+f_r'(0),
  \end{split}
\end{equation*}
where
\begin{equation} \label{eq:bmoag}
  g_{r,z}(w)=\overline{\int_0^z\frac{A(r\zeta)}{1-\overline{w}\zeta}\,d\zeta},\quad w\in\D.
\end{equation}
Since $f_r,g_{r,z}\in H^2$, Lemma~\ref{lm:green2.1} implies
\begin{equation*} \label{eq:bmoau2}
  \begin{split}
    \frac{1}{2\pi}\int_0^{2\pi}f_r(e^{it})\overline{g_{r,z}(e^{it})}\,dt
    &= 2 \int_\D f_r'(w)\overline{g_{r,z}'(w)}\log\frac{1}{|w|}\, dm(w)
    +f_r(0)\overline{g_{r,z}(0)}.
  \end{split}
\end{equation*}
We deduce
\begin{equation*}
  \label{eq:bmoa0}
  |f_r'(z)|^2 \leq 8 \abs{\int_\D f_r'(w)\overline{g_{r,z}'(w)}\log\frac{1}{|w|}\,dm(w)}^2
  +2\abs{f_r(0)\overline{g_{r,z}(0)}-f_r'(0)}^2.
\end{equation*}
By the Hardy-Stein-Spencer formula
\begin{equation*}
  \int_\D\frac{|g_{r,z}'(w)|^2}{|g_{r,z}(w)|}\log\frac{1}{|w|}\,dm(w)
  \leq 2 \, \nm{g_{r,z}}_{H^1},
\end{equation*}
and hence by~\eqref{eq:cnbmoa} and Carleson's theorem~\cite[Theorem~9.3]{D:1970}, there exist
absolute constants $0<C<\infty$ and $0<C'<\infty$ such that
\begin{equation*}
  \label{eq:bmoa2}
  \begin{split}
    \abs{\int_\D f_r'(w)\overline{g_{r,z}'(w)}\log\frac{1}{|w|}\,dm(w)}^2
    &\leq \int_\D\frac{|g_{r,z}'(w)|^2}{|g_{r,z}(w)|}\log\frac{1}{|w|}\,dm(w)\\
    &\qquad \cdot\int_\D|g_{r,z}(w)||f_r'(w)|^2\log\frac{1}{|w|}\,dm(w)\\
    &\leq 2 \, \nm{g_{r,z}}_{H^1}C' \, \nm{\mu_{f_r}}_{\textrm{Carleson}} \, \nm{g_{r,z}}_{H^1}\\
    &=2  C\, \nm{g_{r,z}}_{H^1}^2\nm{f_r}_{\BMOA}^2.
  \end{split}
\end{equation*}
We have $|f_r'(z)|^2 \leq 16C\, \nm{g_{r,z}}_{H^1}^2\nm{f_r}_{\BMOA}^2+4\, |f_r(0)|^2|g_{r,z}(0)|^2+4\, |f_r'(0)|^2$,
and by~\eqref{eq:bmoanorm},
\begin{equation*} \label{eq:bmoaarvio2}
  \begin{split}
    \nm{f_r}_{\BMOA}^2
    &\leq
    64C \, \nm{f_r}_{\BMOA}^2\, 
    \sup_{a\in\D}\int_\D\nm{g_{r,z}}_{H^1}^2(1-|\varphi_a(z)|^2)\,dm(z)\\
    &\qquad + 16\, 
    |f_r(0)|^2 \, \sup_{a\in\D}\int_\D|g_{r,z}(0)|^2(1-|\varphi_a(z)|^2)\,dm(z)
    +16\, |f_r'(0)|^2.
  \end{split}
\end{equation*}
By re-organizing terms and letting $r\to 1^{-}$, the assertion follows.
\end{proof}


\begin{remark} \label{remark:annulusbmoa}
The proof of Theorem~\ref{th:bmoa} shows that, in order to conclude $f\in\BMOA$, it suffices to take the
supremum in \eqref{eq:bmoacond} over any annulus $R<|z|<1$ instead of $\D$.
\end{remark}


\begin{theorem}
\label{th:vmoa}
Let $A\in\mathcal{H}(\D)$ such that
\begin{equation*}
  \lim_{|a|\to 1^-}
  \int_\D\brac{\frac{1}{2\pi}\int_0^{2\pi}\abs{\int_0^z\frac{A(\zeta)d\zeta}{1-e^{-i\theta}\zeta}}dt}^2(1-|\varphi_a(z)|^2)\,dm(z)=0.
\end{equation*}
Then every solution $f$ of~\eqref{eq:de1} belongs to $\VMOA$.
\end{theorem}


\begin{proof}
First, by the assumption and \eqref{eq:Aprim2},
any second primitive of $A$ belongs to $\VMOA$.
Let $f$ be any solution of~\eqref{eq:de1}. By Theorem~\ref{th:bmoa}, Remark~\ref{remark:annulusbmoa},
and Lebesgue's dominated convergence theorem, we deduce $f\in{\BMOA}\subset H^1$; 
compare to the proof of Theorem~\ref{b0short}.
As in the proof of Theorem~\ref{th:bmoa}, we obtain
\begin{equation*}
  |f'(z)|^2\lesssim \nm{g_{r,z}}_{H^1}^2\nm{f}_{\BMOA}^2+|g_{r,z}(0)|^2|f(0)|^2+|f'(0)|^2,\quad z\in\D,
\end{equation*}
where $g_{r,z}$ is the function in~\eqref{eq:bmoag}. Hence, by~\eqref{eq:bmoanorm},
\begin{equation*}
  \begin{split} \label{eq:vmoa}
    \nm{f_a}_{H^2}^2
    &\lesssim\nm{f}_{\BMOA}^2\int_\D\nm{g_{r,z}}_{H^1}^2(1-|\varphi_a(z)|^2)\,dm(z)\\
    &\quad +|f(0)|^2\int_\D\abs{g_{r,z}(0)}^2(1-|\varphi_a(z)|^2)\,dm(z)\\
    &\quad+|f'(0)|^2(1-|a|^2)\int_\D\frac{1-|z|^2}{|1-\overline{a}z|^2}\,dm(z).
  \end{split}
\end{equation*}
The assertion follows by letting $|a|\to 1^-$.
\end{proof}


\section{Hardy spaces} \label{sec:hardy}
For $0<p< \infty$, the Hardy space $H^p$ consists of those $f\in\HH(\D)$ for which
\begin{equation*}
  \nm{f}_{H^p}^p  =\sup_{0\leq r<1}\;\frac{1}{2\pi}\int_0^{2\pi}|f(re^{i\theta})|^p\,d\theta<\infty.
\end{equation*}


\begin{proof}[Proof of Theorem~\ref{prop:main}]
The case $p=2$ follows from the Littlewood-Paley identity by standard estimates, 
and if $k=1$ then much more is true, see~\cite[Theorem~1.2]{P:2009}.

(i) We proceed to prove the following preliminary estimate.
If $0<p<2$, $k\in\N$ and $0<r<1$, then
\begin{equation} \label{eq:preli}
  \|f_r\|_{H^p}^p\lesssim\int_{\D}|f_r(z)|^{p-2}|f_r^{(k)}(z)|^2(1-|z|^2)^{2(k-1)+1}\,dm(z)
  +\frac{\bigg(\sum\limits_{j=0}^{k-1}|f^{(j)}(0)|^p\bigg)^{2/p}}{\|f_r\|_{H^p}^{2-p}}
\end{equation}
for all $f\in\mathcal{H}(\D)$, $f\not\equiv 0$. Here $f_r(z) = f(rz)$ for $z\in\D$.
Write $$d\mu_r(z)=|f_r^{(k)}(z)|^2(1-|z|^2)^{2(k-1)}\,dm(z),$$ for short.
The following argument relies on a classical characterization of $H^p$ spaces which
involves non-tangential approach regions; see \cite[p.~125]{AB:1988}, for example. 
For a fixed $1<\alpha<\infty$, the non-tangential approach region $\Gamma(\zeta)$ 
with vertex at $\zeta\in\T=\partial\D$, of aperture $2\arctan\sqrt{\alpha^2-1}$, is
$\Gamma(\zeta) =\set{z\in\D\,:\, |z-\zeta|\leq\alpha(1-|z|)}$.
The corresponding non-tangential maximal function is given by
\begin{equation*}
  f^\star(\zeta)=\sup_{z\in\Gamma(\zeta)}\;|f(z)|,\quad \zeta\in\T.
\end{equation*}
Fubini's theorem and H\"older's inequality (with indices $2/(2-p)$ and $2/p$) yield
\begin{equation*}
  \begin{split}
    \|f_r\|_{H^p}^p&\asymp\int_{\T}\left(\int_{\Gamma(\z)}d\mu_r(z)\right)^\frac{p}{2}|d\z|
    +\sum_{j=0}^{k-1}|f_r^{(j)}(0)|^p\\
    &\le\int_{\T} f_r^\star(\z)^{(2-p)\frac{p}{2}}\left(\int_{\Gamma(\z)}|f_r(z)|^{p-2}\, d\mu_r(z)\right)^\frac{p}{2}|d\z|
    +\sum_{j=0}^{k-1}|f^{(j)}(0)|^p\\
    &\le\left(\int_{\T} f_r^\star(\z)^p \, |d\z|\right)^\frac{2-p}{2}
    \left(\int_{\T}\int_{\Gamma(\z)}|f_r(z)|^{p-2}\,d\mu_r(z)|d\z|\right)^\frac{p}{2}
    +\sum_{j=0}^{k-1}|f^{(j)}(0)|^p\\
    &\lesssim\|f_r\|_{H^p}^{p(1-\frac{p}{2})}
    \left(\int_{\D}|f_r(z)|^{p-2}(1-|z|^2)\,d\mu_r(z)\right)^\frac{p}{2}
    +\sum_{j=0}^{k-1}|f^{(j)}(0)|^p,
  \end{split}
\end{equation*}
where the last inequality follows from~\cite[pp.~55--56]{G:1981}. Estimate~\eqref{eq:preli} follows by re-organizing the terms.

By a change of variable, we get
\begin{align}
  & \int_{\D}|f_r(z)|^{p-2}|f_r^{(k)}(z)|^2(1-|z|^2)^{2(k-1)+1}\,dm(z)\notag\\
  & \qquad \leq \int_{\D}|f(z)|^{p-2}|f^{(k)}(z)|^2\left(1-|z|^2\right)^{2k-1} \,dm(z). \label{eq:bound}
\end{align}
By means of \eqref{eq:preli} we conclude that, if \eqref{eq:bound} is finite then $f\in H^p$ and
\begin{equation} \label{eq:preli2}
  \|f\|_{H^p}^p\lesssim \int_{\D}|f(z)|^{p-2}|f^{(k)}(z)|^2\left(1-|z|^2\right)^{2k-1} \,dm(z)
  +\frac{\bigg(\sum\limits_{j=0}^{k-1}|f^{(j)}(0)|^p\bigg)^{2/p}}{\|f\|_{H^p}^{2-p}}.
\end{equation}
Cauchy's integral formula, and the estimate $|f(z)|\lesssim \nm{f}_{H^p} (1-|z|^2)^{-1/p}$ for $z\in\D$ \cite[p.~36]{D:1970}, give
$|f^{(j)}(0)|^2 \lesssim \nm{f}^{2-p}_{H^p} \cdot |f^{(j)}(0)|^p$ for $j=0,1,\dotsc, k-1$, which implies
\begin{equation} \label{eq:sum}
  \bigg(\sum_{j=0}^{k-1}|f^{(j)}(0)|^p\bigg)^{2/p} \lesssim \sum_{j=0}^{k-1}|f^{(j)}(0)|^2
  \lesssim \nm{f}^{2-p}_{H^p} \, \sum_{j=0}^{k-1}|f^{(j)}(0)|^p.
\end{equation}
Now \eqref{eq:preli2} and \eqref{eq:sum} prove \eqref{eq:i}.

(ii) Let $2<p<\infty$. Write $q=p-2$ and
$d\mu(z)=|f^{(k)}(z)|^2(1-|z|^2)^{2(k-1)+1}\,dm(z)$, for short. For
$z\in\D$, let $I(z)=\big\{\zeta\in\T :  z\in\Gamma(\zeta) \big\}$
and note that its Euclidean arc length satisfies $|I(z)|\asymp1-|z|$. Fubini's theorem, 
H\"older's inequality (with indices $p/q$ and $p/(p-q)$) and~\cite[pp.~55--56]{G:1981} yield
\begin{equation*}
  \begin{split}
    \int_{\D}|f(z)|^{q}\,d\mu(z)&\asymp\int_{\D}\left(\int_{I(z)}|d\z|\right)\frac{|f(z)|^{q}}{1-|z|^2}\,d\mu(z)
    =\int_{\T}\int_{\Gamma(\z)}\frac{|f(z)|^{q}}{1-|z|^2}\,d\mu(z)\, |d\z|\\
    &\le\left(\int_{\T} f^\star(\z)^p \, |d\z| \right)^\frac{q}{p}
    \left(\int_{\T}\left(\int_{\Gamma(\z)}\frac{d\mu(z)}{1-|z|^2}\right)^\frac{p}{p-q}|d\z|\right)^\frac{p-q}{p}\\
    &\lesssim\|f\|_{H^p}^q\left(\int_{\T}\left(\int_{\Gamma(\z)}|f^{(k)}(z)|^2(1-|z|^2)^{2(k-1)}\,dm(z)\right)^\frac{p}{p-q}|d\z|\right)^\frac{p-q}{p}\\
    &\lesssim\|f\|_{H^p}^{p-2}\Bigg(\|f\|_{H^p}^p-\sum_{j=0}^{k-1}|f^{(j)}(0)|^p\Bigg)^\frac{2}{p}\lesssim\|f\|_{H^p}^p,
  \end{split}
\end{equation*}
and the assertion of (ii) is proved.

(iii) If $f$ is uniformly locally univalent, then $\sup_{z\in\D} \, \left|f''(z)/f'(z) \right| (1-|z|^2)$
is bounded by a constant depending on $\delta$ \cite[Theorem~2]{Y:1977}. Here
$0<\delta\leq 1$ is a constant such that $f$ is univalent in each pseudo-hyperbolic
disc $\Delta_p(z,\delta)$ for $z\in\D$. Since
\begin{equation*}
  \bigg( \frac{f^{(k)}}{f'} \bigg)^\prime = \frac{f^{(k+1)}}{f'} - \frac{f''}{f'} \cdot \frac{f^{(k)}}{f'}, \quad k\in\N,
\end{equation*}
by induction we conclude $\nm{f^{(k+1)} / f'}_{H^\infty_k}< \infty$ for $k\in\N$.
By means of the Hardy-Stein-Spencer formula, we deduce
\begin{equation*}
  \begin{split}
    & \int_\D |f(z)|^{p-2} |f^{(k)}(z)|^2 (1-|z|^2)^{2k-1} \, dm(z) \\
    & \quad \lesssim \bigg\lVert \frac{f^{(k)}}{f'} \bigg\lVert_{H^{\infty}_{k-1}}^2 \int_\D |f(z)|^{p-2} |f'(z)|^2 \log\frac{1}{|z|} \, dm(z)
    \lesssim \nm{f}^p_{H^p},
  \end{split}
\end{equation*}
where the comparison constant depends on $\delta$ and $p$. This concludes the proof of Theorem~\ref{prop:main}.
\end{proof}


\subsection{A class of functions for which Question~\ref{probnorm} has an affirmative answer} \label{sec:zerofree}

If $f\in\mathcal{H}(\D)$ is non-vanishing, then $g = f^{(p-2)/2} f' \in\mathcal{H}(\D)$ and
$g' = \frac{p-2}{2} f^{\frac{p-4}{2}} (f')^2 + f^{\frac{p-2}{2}} f''$.
The Hardy-Stein-Spencer formula~\eqref{eq:hssf} implies
\begin{equation} \label{eq:hssapp1}
  \nm{f}_{H^p}^p
  \leq |f(0)|^p + C_1 \, p^2 \int_{\D} |g(z)|^2 (1-|z|^2) \, dm(z),
\end{equation}
where $0<C_1<\infty$ is an absolute constant.
By standard estimates, there exists an~absolute constant $0<C_2<\infty$ such that
\begin{align*}
  \int_{\D} |g(z)|^2 (1-|z|^2) \, dm(z)
  & \leq  C_2 \left( |g(0)|^2 + \int_{\D} |g'(z)|^2 (1-|z|^2)^3 \, dm(z) \right).
\end{align*}
By~\eqref{eq:hssapp1}, we deduce
\begin{align*}
  \nm{f}_{H^p}^p
  & \leq |f(0)|^p + C_1 C_2 \, p^2 \, \bigg\lVert \frac{f'}{f} \bigg\rVert_{H^\infty_1}^{2-p} \, |f'(0)|^p
    + 2 \, C_1 C_2 \, (p-2)^2 \, \bigg\lVert \frac{f'}{f} \bigg\rVert_{H^\infty_1}^{2} \, \nm{f}_{H^p}^p \\
  & \qquad +  2 C_1 C_2 \, p^2 \int_{\D} |f(z)|^{p-2} |f''(z)|^2 (1-|z|^2)^3 \, dm(z).
\end{align*}

In conclusion, if $f\in\mathcal{H}(\D)$ is non-vanishing and $\nm{f'/f }_{H^\infty_1} = \nm{\log f}_{\B}$ 
is sufficiently small, then \eqref{eq:hpest} holds with $C(p)\asymp p^2$ as $p\to 0^+$.


\subsection{Applications to differential equations} \label{sec:applications}

Theorem~\ref{prop:main} induces an alternative proof for
a special case of \cite[Theorem~1.7]{R:2007}).


\begin{oldtheorem}
\label{th:hp}
Let $0<p\leq 2$ and $A\in\mathcal{H}(\D)$. If \eqref{eq:hpold}
is sufficiently small (depending on~$p$), then all solutions $f$ of~\eqref{eq:de1} satisfy $f\in H^p$.
\end{oldtheorem}


\begin{proof}
Note that
\begin{equation}
  \label{eq:hp}
  \limsup_{r\to 1^-}\,
  \sup_{a\in\D}\,
  \int_{\D} |A(rz)|^2(1-|z|^2)^2 (1-|\varphi_a(z)|^2) \, dm(z)
\end{equation}
is at most a constant multiple of \eqref{eq:hpold}; compare to the proof of Theorem~\ref{th:bmoa2}.
Let $f$ be a~solution of~\eqref{eq:de1}, and $f_r(z)=f(rz)$ for $0<r<1$.
By Theorem~\ref{prop:main}(i), we deduce
\begin{align*}
  \nm{f_r}_{H^p}^p & \lesssim \int_{\D} |f_r(z)|^{p-2} r^2 |f''(rz)|^2 (1-|z|^2)^3 \, dm(z) + |f(0)|^p + |f'(0)|^p\\
                   &  \lesssim \int_{\D} |f_r(z)|^p \, |A(rz)|^2 (1-|z|^2)^3 \, dm(z) + |f(0)|^p + |f'(0)|^p
\end{align*}
for any $0<p\leq 2$. If \eqref{eq:hp} is sufficiently small, then
Carleson's theorem~\cite[Theorem~9.3]{D:1970} implies that
$\nm{f_r}_{H^p}$ is uniformly bounded for all sufficiently large $0<r<1$.
By letting $r\to 1^{-}$, we obtain $f\in H^p$.
\end{proof}

An argument similar to the one above, taking advantage of Theorem~\ref{prop:main}(i),
leads to a characterization of $H^p$ solutions of \eqref{eq:de1}:
if $0<p\leq 2$, $f$ is a solution of~\eqref{eq:de1} and $d\mu_A(z)=|A(z)|^2(1-|z|^2)^3\, dm(z)$ is a~Carleson measure,
then $f\in H^p$ if and only if
\begin{equation} \label{eq:ee}
  \int_\D |f(z)|^p \, d\mu_A(z) < \infty.
\end{equation}
For example, if $f$ is a normal (in the sense of Lehto and Virtanen)
solution of~\eqref{eq:de1} and $\mu_A$ is a Carleson measure, then \eqref{eq:ee} holds
for all sufficiently small $0<p<\infty$ by~\cite[Corollary~9]{GNR:preprint}.


\begin{remark}
If Question~\ref{probnorm} had an affirmative answer, then
Theorem~\ref{th:hp} would admit the following immediate improvement:
if $A\in H(\D)$ such that \eqref{eq:hp} is finite, then
any solution~$f$ of~\eqref{eq:de1} satisfies $f\in \bigcup_{0<p<\infty} H^p$.
\end{remark}


\end{document}